\documentclass[a4paper,leqno,draft]{article}
\usepackage{amssymb,amsmath,amsfonts,amsthm,
mathrsfs}
\newtheorem{theorem}{Theorem}[section]
\newtheorem{corollary}[theorem]{Corollary}

\newtheorem{proposition}[theorem]{Proposition}
\newtheorem{definition}[theorem]{Definition}
\theoremstyle{definition}
\newtheorem{remark}[theorem]{Remark}
\newtheorem{example}[theorem]{Example}

\newcommand{\wt}[1]{\widetilde{#1}}

\newcommand{\Cinf}{\ensuremath{\mathcal{C}^\infty}}
\newcommand{\Cinfc}{\ensuremath{\mathcal{C}^\infty_{\text{c}}}}
\newcommand{\D}{\ensuremath{{\cal D}}}
\renewcommand{\S}{\mathscr{S}}
\newcommand{\E}{\ensuremath{{\cal E}}}
\newcommand{\LL}{\mathcal{L}}
\newcommand{\mM}{\mathcal{M}}

\newcommand{\mb}[1]{\ensuremath{\mathbb{#1}}}
\newcommand{\N}{\mb{N}}

\newcommand{\R}{\mb{R}}
\newcommand{\C}{\mb{C}}

\newcommand{\G}{\ensuremath{{\cal G}}}
\newcommand{\Gt}{\ensuremath{{\cal G}_\tau}}

\newcommand{\Gc}{\ensuremath{{\cal G}_\mathrm{c}}}
\newcommand{\Gcinf}{\ensuremath{{\cal G}^\infty_\mathrm{c}}}
\newcommand{\GS}{\G_{{\, }\atop{\hskip-4pt\scriptstyle\S}}\!}
\newcommand{\EM}{\ensuremath{{\cal E}_{M}}}

\newcommand{\Neg}{\mathcal{N}}

\newcommand{\Ginf}{\ensuremath{\G^\infty}}
\newcommand{\Gtinf}{\mathcal{G}^{\infty}_\tau}
\newcommand{\GSinf}{\G^\infty_{{\, }\atop{\hskip-3pt\scriptstyle\S}}}

\newcommand{\lara}[1]{\langle #1 \rangle}

\newcommand{\supp}{\mathrm{supp}}
\newcommand{\ssc}{\mathrm{sc}}

\newcommand{\val}{\mathrm{v}} 
\newcommand{\esp}{\mathrm{e}}

\newcommand{\beq}{\begin{equation}}
\newcommand{\eeq}{\end{equation}}

\newcommand{\eps}{\varepsilon}

\newcommand{\Om}{\Omega}

\newcommand{\mP}{\mathcal{P}}

\newcommand{\LLb}{\mathcal{L}_{\rm{b}}}

\newcommand{\dslash}{d\hspace{-0.4em}{ }^-\hspace{-0.2em}}

\begin{document}

\title{{\bf Sufficient conditions of local solvability for partial differential operators in the Colombeau context}}

\author{Claudia Garetto\footnote{Supported by FWF (Austria), grant T305-N13.}\\
Institut f\"ur Grundlagen der Bauingenieurwissenschaften\\
Leopold-Franzens-Universit\"at Innsbruck\\
Technikerstrasse 13, 6020 Innsbruck\\
\texttt{claudia@mat1.uibk.ac.at}\\
}
\date{}
\maketitle

\begin{abstract} 
We provide sufficient conditions of local solvability for partial differential operators with variable Colombeau coefficients. We mainly concentrate on operators which admit a right generalized pseudodifferential parametrix and on operators which are a bounded perturbation of a differential operator with constant Colombeau coefficients. The local solutions are intended in the Colombeau algebra $\G(\Om)$ as well as in the dual $\LL(\Gc(\Om),\wt{\C})$.
\end{abstract}

{\bf{Key words:}} algebras of generalized functions, generalized solutions of partial differential equations

\emph{AMS 2000 subject classification: 46F30, 35D99}

\setcounter{section}{-1}
\section{Introduction}
Colombeau algebras of generalized functions \cite{Colombeau:85, GKOS:01} have proved to be a well-organized and powerful framework where to solve linear and nonlinear partial differential equations involving non-smooth coefficients and strongly singular data. So far, the purpose of many authors has been to find for a specific problem of applicative relevance the most suitable Colombeau framework where first to provide solvability and second to give a qualitative description of the solutions. We recall that several results of existence and uniqueness of the solution have been obtained in this generalized context for hyperbolic Cauchy problems with singular coefficients and initial data \cite{GH:03, HdH:01, HdH:01c, LO:91, O:89}, for elliptic and hypoelliptic equations \cite{GGO:03, HO:03, HOP:05} and for divergent type quasilinear Dirichlet problems with singularities \cite{PilSca:06}. The setting of generalized functions employed is the Colombeau algebra $\G(\Om)$ constructed on an open subset $\Om$ of $\R^n$, or more in general the Colombeau space $\G_E$ of generalized functions based on a locally convex topological vector space $E$ (see \cite{Garetto:05b, Garetto:05a} for definitions and properties). For instance in \cite{BO:92, HO:03} solvability is provided in the Colombeau space based on $H^\infty(\R^n)=\cap_{s\in\R}H^s(\R^n)$. Recently, in order to enlarge the family of generalized hyperbolic problems which can be solved and in order to provide a more refined microlocal investigation of the qualitative properties of the solution, the dual $\LL(\Gc(\Om),\wt{\C})$ has replaced the classical Colombeau setting $\G(\Om)$ \cite{O:07}.
 
This paper is devoted to the general problem of solvability or more precisely local solvability in the Colombeau context for partial differential operators with Colombeau coefficients. Namely, it is the starting point of a challenging project which aims to discuss and fully understand solvability and local solvability of partial differential operators in the Colombeau context. 

Instead of dealing with a specific equation $P(x,D)u=v$ and looking for a new setting of generalized functions tailored to this particular problem, we want to determine a class of locally solvable partial differential operators. This will be done by finding some sufficient conditions on $P$ of local solvability in the Colombeau context $\G(\Om)$ or $\LL(\Gc(\Om),\wt{\C})$.
As in the classical theory of partial differential operators with smooth coefficients, mainly developed by H\"ormander in \cite{Hoermander:63, Hoermander:V1-4}, different mathematical methods and level of technicalities concern the investigation of solvability when the coefficients are constant or not. The Malgrange-Ehrenpreis theorem essentially reduces the solvability issue to the search for a fundamental solution in the constant coefficients case but clearly this powerful tool loses efficiency when the coefficients are variable. In this situation indeed, not only the structural properties of the operator but also the geometric features of the set $\Om$ where we want to solve the equation play a relevant role in stating existence theorems of local or global solvability. 

Differential operators with constant Colombeau coefficients, i.e. coefficients in the ring $\wt{\C}$ of complex generalized numbers, have been studied by various authors \cite{Garetto:08b, Garetto:08c, HO:03}. In particular a notion of fundamental solution has been introduced in \cite{Garetto:08b} as a functional in the dual $\LL(\Gc(\R^n),\wt{\C})$ providing, by means of a generalized version of the Malgrange-Ehrenpreis theorem, a straightforward result of solvability in the Colombeau context. In detail, a solution to the equation $P(D)u=v$, $P(D)=\sum_{|\alpha|\le m}c_\alpha D^\alpha$ with $c_\alpha\in\wt{\C}$ has been obtained via convolution of the right hand side $v$ with a fundamental solution $E$ and certain regularity qualities of the operator $P(D)$, the $\G$- and $\Ginf$-hypoellipticity for instance, have been proven to be equivalent to some structural properties of its fundamental solutions \cite[Theorems 3.6, 4.2]{Garetto:08c}.

In this paper we concentrate on differential operators with variable Colombeau coefficients, i.e. $P(x,D)=\sum_{|\alpha|\le m}c_\alpha(x)D^\alpha$. Being aware of the objective difficulty of investigating solvability in wide generality, we fix our attention on two classes of operators: the operators which are approximately invertible, in the sense that they admit a right generalized pseudodifferential parametrix, and the operators which are locally a bounded perturbation of a differential operator with constant Colombeau coefficients. In both these cases we will formulate sufficient conditions of solvability which will require suitable assumptions on the moderateness properties of the coefficients and the right-hand side. Note that the symbolic calculus for generalized pseudodifferential operators developed in \cite{Garetto:04, Garetto:ISAAC07, GGO:03} is essential for studying the first class of operators whereas the theory of fundamental solutions in the dual $\LL(\Gc(\R^n),\wt{\C})$ is heavily used in finding a local Colombeau solution for operators of bounded perturbation type. This paper can therefore be considered a natural follow-up of \cite{Garetto:08b}.

We now describe the contents of the paper in more detail.

Section \ref{sec_basic} collects the needed background of Colombeau theory and recalls, for the advantage of the reader, the results of solvability obtained in the generalized constant coefficients case. Definition and properties of a fundamental solution in the dual $\LL(\Gc(\R^n),\wt{\C})$ are the topic of Subsection \ref{subsec_fund}. Inspired by the work of H\"ormander in \cite[Section 10.4]{Hoermander:V2}, in Subsection \ref{subsec_comp} we introduce an order relation between operators with constant Colombeau coefficients in terms of the corresponding weight functions. In other words, we make use of the weight function $\wt{P}(\xi)=\big(\sum_{|\alpha|\le m}|\partial^\alpha P(\xi)|^2)^{\frac{1}{2}}$ (with values in $\wt{\R}$) in order to determine the differential operators which are stronger (or weaker, respectively) than $P(D)$. By stating this notion in few equivalent ways (Proposition \ref{prop_basic_2}) we prove that an $m$-oder differential operator $P(D)$ with coefficients in $\wt{\C}$ is stronger than any differential operator with coefficients in $\wt{\C}$ of order $\le m$ if and only if it is $\G$-elliptic. Analogously, we prove that if $P(D)$ is of principal type then it is stronger or better it dominates (Definition \ref{def_domin}) any differential operator with order $\le m-1$. These results of comparison among differential operators with constant Colombeau coefficients will be used in Sections \ref{sec_bounded} and \ref{sec_local}. 

In Section \ref{sec_parametrix} we begin our investigation of local solvability in the Colombeau context by considering differential operators with coefficients in $\G$ that admit a right generalized pseudodifferential operator parametrix. Given $P(x,D)$ this means that there exists a pseudodifferential operator $q(x,D)$ such that the equality $P_\eps(x,D)q_\eps(x,D)=I+r_\eps(x,D)$ holds at the level of representatives with $(r_\eps)_\eps$ a net of regularizing operators. The moderateness properties of the reminder term $r_\eps$ are crucial in determining for each $x_0\in\R^n$ a sufficiently small neighborhood $\Om$ such that the equation $P(x,D)T=F$ on $\Om$ is solvable in $\LL(\Gc(\Om),\wt{\C})$ for any $L^2_{\rm{loc}}$-moderate functional $F$. Different notions of a generalized hypoelliptic symbol have been introduced in the recent past in \cite{Garetto:ISAAC07, GGO:03, GH:05}. They all assure the existence of a generalized parametrix $q(x,D)$ but in general do not guarantee the moderateness properties on the regularizing operator $r(x,D)$ which are essential for the previous result of local solvability. For this reason in Propositions \ref{prop_sol_hyp}, \ref{prop_sol_hyp_2} and \ref{prop_sol_hyp_3} we make use of a definition of generalized hypoelliptic symbol, first presented in \cite{Garetto:04}, which is less general than the ones considered in  \cite{Garetto:ISAAC07, GGO:03, GH:05}, but that combining the right moderateness and regularity properties, provides local solvability in $\LL(\Gc(\Om),\wt{\C})$ as well as in $\G(\Om)$ and $\Ginf(\Om)$.

Section \ref{sec_elliptic} deals with a special class of differential operators: the operators which are $\G$-elliptic in a neighborhood of a point $x_0$. Since they have a generalized hypoelliptic symbol they admit a local generalized parametrix and from the statements of Section \ref{sec_parametrix} we easily obtain results of local solvability. The most interesting fact is that this locally solvable operators are actually a perturbation of a differential operator with constant Colombeau coefficients, namely the same operator evaluated at $x=x_0$. Using the concepts of Subsection \ref{subsec_comp} we prove that a differential operator $P(x,D)$ which is $\G$-elliptic in $x_0$ can be written in the form \beq
\label{BP_intro}
P_0(D)+\sum_{j=1}^r c_j(x)P_j(D),
\eeq
where $P_0(D)=P(x_0,D)$, the operators $P_j(D)$ have coefficients in $\wt{\C}$ and are all weaker than $P_0(D)$ and the generalized functions $c_j$ belong to the Colombeau algebra $\G(\R^n)$. This fact motivates our interest for the wider class of generalized differential operators which are locally a bounded perturbation of a differential operator with constant Colombeau coefficients as in \eqref{BP_intro}. The precise definition and some first examples are the topic of Section \ref{sec_bounded}.

In Section \ref{sec_local} we provide some sufficient conditions of local solvability for operators of bounded perturbation type as defined in Section \ref{sec_bounded}. The local solutions are obtained by using a fundamental solution in $\LL(\Gc(\R^n),\wt{\C})$ of $P_0(D)$, the comparison between the operators $P_j(D)$ and $P_0(D)$ and, at the level of representatives, suitable estimates of $B_{p,k}$-moderateness. Theorem \ref{theo_locsolv_easy} and Theorem \ref{theo_locsolv} have H\"ormander's theorem of local solvability for operators of constant strength (\cite[Theorem 7.3.1]{Hoermander:63}, \cite[Theorem 13.3.3]{Hoermander:V2}) as a blueprint.

The paper ends with a sufficient condition of local solvability for operators which are not necessarily of bounded perturbation type or do not have a generalized parametrix. In Section \ref{sec_pseudo}, inspired by \cite[Chapter 4]{SaintRaymond:91}, we prove that a certain Sobolev estimate from below on the adjoint of a generalized pseudodifferential operator is sufficient to obtain local solvability in the dual $\LL(\Gc(\Om),\wt{\C})$. The proof has the interesting feature of using the theory of generalized Hilbert $\wt{\C}$-modules developed in \cite{GarVer:08} and in particular the projection theorem on an internal subset. We finally give some examples of generalized differential and pseudodifferential operators fulfilling this sufficient condition. 

\section{Colombeau theory and partial differential operators with constant Co\-lom\-be\-au coefficients}
\label{sec_basic}
In this section we recall some basic notions of Colombeau theory and, for the advantage of the reader, what has been proved in \cite{Garetto:08b, HO:03} about solvability in the Colombeau context of partial differential operators with generalized constant coefficients. 
\subsection{Basic notions of Colombeau theory}
Main sources of this subsection are \cite{Colombeau:85, Garetto:05b, Garetto:05a, GGO:03, GH:05, GKOS:01}.
\paragraph{Nets of numbers.}
Before dealing with the major points of the Colombeau construction we begin by recalling some definitions concerning elements of $\mathbb{C}^{(0,1]}$.

A net $(u_\eps)_\eps$ in $\C^{(0,1]}$ is said to be \emph{strictly nonzero} if there exist $r>0$ and $\eta\in(0,1]$ such that $|u_\eps|\ge \eps^r$ for all $\eps\in(0,\eta]$. The regularity issues discussed in this paper will make use of the following concept of \emph{slow scale net (s.s.n)}. A slow scale net is a net $(r_\eps)_\eps\in\C^{(0,1]}$ such that 
\[
\forall q\ge 0\, \exists c_q>0\, \forall\eps\in(0,1]\qquad\qquad|r_\eps|^q\le c_q\eps^{-1}.
\]
\paragraph{Colombeau spaces based on $E$.}
Let $E$ be a locally convex topological vector space topologized through the family of seminorms $\{p_i\}_{i\in I}$. The elements of 
\[
\begin{split} 
\mM_E &:= \{(u_\eps)_\eps\in E^{(0,1]}:\, \forall i\in I\,\, \exists N\in\N\quad p_i(u_\eps)=O(\eps^{-N})\, \text{as}\, \eps\to 0\},\\
\mM^\ssc_E &:=\{(u_\eps)_\eps\in E^{(0,1]}:\, \forall i\in I\,\, \exists (\omega_\eps)_\eps\, \text{s.s.n.}\quad p_i(u_\eps)=O(\omega_\eps)\, \text{as}\, \eps\to 0\},\\
\Neg_E &:= \{(u_\eps)_\eps\in E^{(0,1]}:\, \forall i\in I\,\, \forall q\in\N\quad p_i(u_\eps)=O(\eps^{q})\, \text{as}\, \eps\to 0\},
\end{split}
\]
are called $E$-moderate, $E$-moderate of slow scale type and $E$-negligible, respectively. We define the space of \emph{generalized functions based on $E$} as the factor space $\G_E := \mM_E / \Neg_E$. The expression \emph{``of slow scale type''} is used for the generalized functions of the factor space $\G^\ssc_E:=\mM^\ssc_E/\Neg_E$. The elements of $\G_E$ are equivalence classes for which we use the notation $u=[(u_\eps)_\eps]$.  

Let $\Om$ be an open subset of $\R^n$. For coherence with the notations already in use, we set $\EM(\Om)=\mM_{\E(\Om)}$, $\Neg(\Om)=\Neg_{\E(\Om)}$, $\EM=\mM_{\C}$ and $\Neg=\Neg_{\C}$. The Colombeau algebra $\G(\Om)$, as originally defined in its full version by Colombeau in \cite{Colombeau:85}, is obtained as the space $\G_E$ with $E=\E(\Om)$. Analogously, the rings $\wt{\C}$ and $\wt{\R}$ of complex and real generalized numbers are the Colombeau spaces $\G_\C$ and $\G_\R$ respectively. $\wt{\C}$ is also the set of constants of $\G(\R^n)$. The space of distributions $\D'(\Om)$ is embedded into $\G(\Om)$ via convolution with a mollifier (see \cite{GKOS:01} for more details). Since $\G(\Om)$ is a sheaf with respect to $\Om$ one has a notion of support for $u\in\G(\Om)$ and a subalgebra $\Gc(\Om)$ of compactly supported generalized functions.
\paragraph{Regularity theory.} Regularity theory in the Colombeau context as initiated in \cite{O:92} is based on the subalgebra $\Ginf(\Om)$ of all elements $u$ of $\G(\Om)$ having a representative $(u_\eps)_\eps$ belonging to the set
\[
\EM^\infty(\Om):=\{(u_\eps)_\eps\in\E[\Om]:\ \forall K\Subset\Om\, \exists N\in\N\, \forall\alpha\in\N^n\quad\sup_{x\in K}|\partial^\alpha u_\eps(x)|=O(\eps^{-N})\, \text{as $\eps\to 0$}\}.
\]
$\Ginf(\Om)$ coincides with the factor space $\EM^\infty(\Om)/\Neg(\Om)$ and by construction has the intersection property $\Ginf(\Om)\cap \D'(\Om)=\Cinf(\Om)$.
\paragraph{Topological theory of Colombeau spaces.}
The family of seminorms $\{p_i\}_{i\in I}$ on $E$ determines a \emph{locally convex $\wt{\C}$-linear} topology on $\G_E$ (see \cite[Definition 1.6]{Garetto:05b}) by means of the \emph{valuations}
\[
\val_{p_i}([(u_\eps)_\eps]):=\val_{p_i}((u_\eps)_\eps):=\sup\{b\in\R:\qquad p_i(u_\eps)=O(\eps^b)\, \text{as $\eps\to 0$}\}
\] 
and the corresponding \emph{ultra-pseudo-seminorms} $\{\mP_i\}_{i\in I}$, where $\mP_i(u)=\esp^{-\val_{p_i}(u)}$. For the sake of brevity we omit to report definitions and properties of valuations and ultra-pseudo-seminorms in the abstract context of $\wt{\C}$-modules. Such a theoretical presentation can be found in \cite[Subsections 1.1, 1.2]{Garetto:05b}. More in general a theory of topological and locally convex topological $\wt{\C}$-modules has been developed in \cite{Garetto:05a}. The Colombeau algebra $\G(\Om)$ has the structure of a Fr\'echet $\wt{\C}$-modules and $\Gc(\Om)$ is the inductive limit of a family of Fr\'echet $\wt{\C}$-modules. We recall that on $\wt{\C}$ the valuation and the ultra-pseudo-norm obtained through the absolute value in $\C$ are denoted by $\val_{\wt{\C}}$ and $|\cdot|_{\esp}$ respectively.
\paragraph{The dual $\LL(\Gc(\Om),\wt{\C})$ and its basic functionals.}
$\LL(\Gc(\Om),\wt{\C})$ is the set of all continuous $\wt{\C}$-linear functionals on $\Gc(\Om)$. As proven in \cite{Garetto:05b} it contains (via continuous embedding) both the algebras $\Ginf(\Om)$ and $\G(\Om)$, i.e., $\Ginf(\Om)\subseteq\G(\Om)\subseteq\LL(\Gc(\Om),\wt{\C})$. The inclusion $\G(\Om)\subseteq\LL(\Gc(\Om),\wt{\C})$ is given via integration ($u\to\big( v\to\int_\Om u(x)v(x)dx\big)$, for definitions and properties of the integral of a Colombeau generalized functions see \cite{GKOS:01}). A special subset of $\LL(\Gc(\Om),\wt{\C})$ is obtained by requiring the so-called ``basic'' structure. In detail, we say that $T\in\LL(\Gc(\Om),\wt{\C})$ is basic (or equivalently $T\in\LLb(\Gc(\Om),\wt{\C})$) if there exists a net $(T_\eps)_\eps\in\D'(\Om)^{(0,1]}$ fulfilling the following condition: for all $K\Subset\Om$ there exist $j\in\N$, $c>0$, $N\in\N$ and $\eta\in(0,1]$ such that
\[
\forall f\in\D_K(\Om)\, \forall\eps\in(0,\eta]\qquad\quad
|T_\eps(f)|\le c\eps^{-N}\sup_{x\in K,|\alpha|\le j}|\partial^\alpha f(x)|
\]
and $Tu=[(T_\eps u_\eps)_\eps]$ for all $u\in\Gc(\Om)$.

Analogously one can introduce the dual $\LL(\G(\Om),\wt{\C})$ and the corresponding set $\LLb(\G(\Om),\wt{\C})$ of basic functionals. As in distribution theory, Theorem 1.2 in \cite{Garetto:05b} proves that $\LL(\G(\Om),\wt{\C})$ can be identified with the set of functionals in $\LL(\Gc(\Om),\wt{\C})$ having compact support.
\paragraph{Generalized differential operators.} $\G(\Om)$ is a differential algebra, in the sense that derivatives of any order can be defined extending the corresponding distributional ones. We can therefore talk of differential operators in the Colombeau context or, for simplicity, of \emph{generalized differential operators}. Clearly, a differential operator with singular distributional coefficients generates a differential operator in the Colombeau context by embedding its coefficients in the Colombeau algebra. Let  
\[
P(x,D)=\sum_{|\alpha|\le m}c_\alpha(x)D^\alpha,
\] 
with $c_\alpha\in\G(\Om)$ for all $\alpha$. Its symbol
\[
P(x,\xi)=\sum_{|\alpha|\le m}c_\alpha(x)D^\alpha
\]
is a polynomial of order $m$ with coefficients in $\G(\Om)$ and representatives
\[
P_\eps(x,\xi)=\sum_{|\alpha|\le m}c_{\alpha,\eps}(x)D^\alpha.
\]
The operator $P(x,D)$ maps $\Gc(\Om)$, $\G(\Om)$ and $\LL(\Gc(\Om),\wt{\C})$ into themselves respectively and $\Ginf(\Om)$ into $\Ginf(\Om)$ if the coefficients are $\Ginf$-regular. When the coefficients are constant ($c_\alpha\in\wt{\C}$ for all $\alpha$) we use the notation $P(D)$.
 
\subsection{Fundamental solutions in $\LLb(\Gc(\R^n),\wt{\C})$}
\label{subsec_fund}
Let $P(D)$ be a partial differential operator of order $m$ with coefficients in $\wt{\C}$. Any net of polynomials $(P_\eps)_\eps$ determined by a choice of representatives of the coefficients of $P(D)$ is called a representative of the polynomial $P$. Consider the weight function $\wt{P}:\R^n\to\wt{\R}$ defined by
\[
\wt{P}^2(\xi)=\sum_{|\alpha|\le m}|\partial^\alpha P(\xi)|^2.
\]
The arguments in \cite[(2.1.10)]{Hoermander:63} yield the following assertion: there exists $C>0$ depending only on $m$ and $n$ such that for all $(P_\eps)_\eps$ the inequality
\beq
\label{est_Hoer}
\wt{P_\eps}(\xi+\eta)\le (1+C|\xi|)^m\wt{P_\eps}(\eta)
\eeq
is valid for all $\xi,\eta\in\R^n$ and all $\eps\in(0,1]$. When the function $\wt{P}:\R^n\to\wt{\R}$ is invertible in some point $\xi_0$ of $\R^n$ Lemma 7.5 in \cite{HO:03} proves that for all representative $(P_\eps)_\eps$ of $P$ there exist $N\in\N$ and $\eta\in(0,1]$ such that
\beq
\label{est_inv}
\wt{P_\eps}(\xi)\ge \eps^N(1+C|\xi_0-\xi|)^{-m},
\eeq
for all $\xi\in\R^n$ and $\eps\in(0,\eta]$. This means that $\wt{P}$ is invertible in any $\xi$ once it is invertible in some $\xi_0$. Note that the constant $C>0$ is the same appearing in \eqref{est_Hoer} and $\eps^N$ comes from the invertibility in $\wt{\R}$ of $\wt{P}(\xi_0)$. It is not restrictive to assume for some strictly non-zero net $(\lambda_\eps)_\eps$ that
\[
\wt{P_\eps}(\xi)\ge \lambda_\eps(1+C|\xi_0-\xi|)^{-m},
\]
for all $\eps\in(0,1]$.

In the sequel $\mathcal{K}$ is the set of tempered weight functions introduced by H\"ormander in
\cite[Definition 2.1.1]{Hoermander:63}, i.e., the set of all positive functions $k$ on $\R^n$ such that for some
constants $C>0$ and $N\in\N$ the inequality
\[
\label{ineq_weight}
k(\xi+\eta)\le (1+C|\xi|)^N k(\eta)
\]
holds for all $\xi,\eta\in\R^n$. Concerning the H\"ormander spaces $B_{p,k}$ which follow, main references are \cite{Hoermander:63, Hoermander:V2}. Typical example of a weight function is $k(\xi)=\lara{\xi}^s=(1+|\xi|^2)^{\frac{s}{2}}$, $s\in\R$.
\begin{definition}
\label{def_cla_bpk}
If $k\in\mathcal{K}$ and $p\in[1,+\infty]$ we denote by $B_{p,k}(\R^n)$ the set of all distributions $w\in\S'(\R^n)$ such that $\widehat{w}$ is a function and
\[
\Vert w\Vert_{p,k}=(2\pi)^{-n}\Vert k\widehat{w}\Vert_p <	\infty.
\]
\end{definition}
$B_{p,k}(\R^n)$ is a Banach space with the norm introduced in Definition \ref{def_cla_bpk}. We have $\S(\R^n)\subset B_{p,k}(\R^n)\subset\S'(\R^n)$ (in a topological sense) and that $\Cinfc(\R^n)$ is dense in $B_{p,k}(\R^n)$ for $p<\infty$.

The inequality \eqref{est_Hoer} says that $\wt{P_\eps}$ is a tempered weight function for each $\eps$ so it is meaningful to consider the sets $B_{\infty,\wt{P_\eps}}(\R^n)$ of distributions as we will see in the next theorem, proven in \cite{Garetto:08b}.
\begin{theorem}
\label{theo_fund_P}
To every differential operator $P(D)$ with coefficients in $\wt{\C}$ such that $\wt{P}(\xi)$ is invertible in some $\xi_0\in\R^n$ there exists a fundamental solution $E\in\LLb(\Gc(\R^n),\wt{\C})$. More precisely, to every $c>0$ and $(P_\eps)_\eps$ representative of $P$ there exists a fundamental solution $E$ given by a net of distributions $(E_\eps)_\eps$ such that $E_\eps/\cosh(c|x|)\in B_{\infty,\wt{P_\eps}}(\R^n)$ and for all $\eps$
\[
\biggl\Vert \frac{E_\eps}{\cosh(c|x|)}\biggr\Vert_{\infty,\wt{P_\eps}}\le C_0,
\]
where the constant $C_0$ depends only on $n,m$ and $c$.
\end{theorem}
One sees in the proof of Theorem \ref{theo_fund_P} (Proposition 3.5 and Theorem 3.3 in \cite{Garetto:08c}) that for each $\eps$ the distribution $E_\eps$ is a fundamental solution of the operator $P_\eps(D)$. Theorem \ref{theo_fund_P} entails the following solvability result.
\begin{theorem}
\label{theom_solv_dual_1}
Let $P(D)$ be a partial differential operator with coefficients in $\wt{\C}$ such that $\wt{P}$ is invertible in some $\xi_0\in\R^n$. Then the equation
\beq
\label{eq_1_P}
P(D)u=v
\eeq
\begin{itemize}
\item[(i)] has a solution $u\in\G(\R^n)$ if $v\in\Gc(\R^n)$,
\item[(ii)] has a solution $u\in\Ginf(\R^n)$ if $v\in\Gcinf(\R^n)$,
\item[(iii)] has a solution $u\in\LL(\Gc(\R^n),\wt{\C})$ if $v\in\LL(\G(\R^n),\wt{\C})$,
\item[(iv)] has a solution $u\in\LLb(\Gc(\R^n),\wt{\C})$ if $v\in\LLb(\G(\R^n),\wt{\C})$.
\end{itemize}
\end{theorem}
Theorem \ref{theom_solv_dual_1} extends to the dual the solvability result obtained in $\G$ by H\"ormann and Oberguggenberger in \cite{HO:03}. A more detailed investigation of the properties of $u$, which heavily makes use of the theory of $B_{p,k}$-spaces, can be found in \cite[Appendix]{Garetto:08b}. 
\begin{remark}
\label{rem_HO_inv}
The condition of invertibility of $\wt{P}$ in a point $\xi_0$ of $\R^n$ turns out to be equivalent to the solvability statement $(i)$ of Theorem \ref{theom_solv_dual_1}. More precisely, Theorem 7.8 in \cite{HO:03} shows that if $v$ is invertible in some point of $\Om$ and the equation $P(D)u=v$ is solvable in $\G(\Om)$ then $\wt{P}$ is invertible in some point of $\R^n$. In the same paper the authors prove that the invertibility of the principal symbol $P_m$ in some $\xi_0$ implies the invertibility of $\wt{P}(\xi_0)$. The converse does not hold as one can see from $P_\eps(\xi)=a_\eps\xi+i$, with $a=[(a_\eps)_\eps]\neq 0$ real valued and not invertible. The principal symbol $P_1$ is not invertible (in any point of $\R^n$) but $\wt{P}^2(0)=1+a^2$ is invertible in $\wt{\R}$. In the same way we have that the existence of an invertible coefficient in the principal part of $P(D)$ is a sufficient but not necessary condition for the invertibility of the weight function $\wt{P}$. Note that there exist differential operators where all the coefficients are not invertible which still have an invertible weight function. An example is given by $P_\eps(\xi_1,\xi_2)=a_\eps\xi_1+ib_\eps\xi_2$, where $a_\eps=1$ if $\eps=n^{-1}$, $n\in\N$, and $0$ otherwise, and $b_\eps=0$ if $\eps=n^{-1}$, $n\in\N$, and $1$ otherwise. The coefficients generated by $(a_\eps)_\eps$ and $(b_\eps)_\eps$ are clearly not invertible but $\wt{P_\eps}^2(1,1)=2(a_\eps^2+b_\eps^2)=2$.
\end{remark}

\subsection{Comparison of differential operators with constant Colombeau coefficients}
\label{subsec_comp}
Inspired by \cite[Section 10.4]{Hoermander:V2} we introduce an order relation between operators with constant Colombeau coefficients by comparing the corresponding weight functions. 
\begin{definition}
\label{def_ord_rel}
Let $P(D)$ and $Q(D)$ be partial differential operators with coefficients in $\wt{\C}$. We say that $P(D)$ is stronger than $Q(D)$ ($Q(D)\prec P(D)$) if there exist representatives $(P_\eps)_\eps$ and $(Q_\eps)_\eps$ and a moderate net $(\lambda_\eps)_\eps$ such that
\[
\wt{Q_\eps}(\xi)\le \lambda_\eps \wt{P_\eps}(\xi)
\]
for all $\xi\in\R^n$ and $\eps\in(0,1]$
\end{definition}
In the sequel we collect some estimates valid for polynomials with coefficients in $\C$ proven in \cite[Theorem 10.4.1, Lemma 10.4.2]{Hoermander:V2}. We recall that $\wt{Q}(\xi,t)$ denotes the function $(\sum_\alpha |Q^{(\alpha)}(\xi)|^2t^{2|\alpha|})^{\frac{1}{2}}$ for $Q$ polynomial of degree $\le m$ in $\R^n$ and $t$ positive real number. Clearly $\wt{Q}(\xi,1)=\wt{Q}(\xi)$.
\begin{proposition}
\label{prop_dom}
\leavevmode
\begin{itemize}
\item[(i)] There exists a constant $C>0$ such that for every polynomial $Q$ of degree $\le m$ in $\R^n$,
\beq
\label{Hoer_1}
\frac{\wt{Q}(\xi,t)}{C}\le \sup_{|\eta|<t}|Q(\xi+\eta)|\le C\wt{Q}(\xi,t)
\eeq
for all $\xi\in\R^n$ and $t>0$.
\item[(ii)] There exist constants $C'$ and $C''$ such that for all polynomials $P$ and $Q$ of degree $\le m$,
\beq
\label{Hoer_2}
C'\wt{P}(\xi)\wt{Q}(\xi)\le \wt{PQ}(\xi)\le C''\wt{P}(\xi)\wt{Q}(\xi)
\eeq
for all $\xi\in\R^n$.
\end{itemize}
\end{proposition}
\begin{proof}
The first assertion is Lemma 10.4.2 in \cite{Hoermander:V2}. Concerning $(ii)$ the second inequality is clear from Leibniz'rule. From \eqref{Hoer_1} for any polynomial $Q$ we find $\eta$ with $|\eta|\le 1$ such that 
\beq
\label{est_dom}
\wt{Q}(\xi)\le C|Q(\xi+\eta)|\le C\wt{Q}(\xi+\eta)
\eeq
for all $\xi$. Since from \cite[(2.1.10)]{Hoermander:63} there exists a constant $C_0$ depending only on $m$ and $n$ such that $\wt{Q}(\xi+\theta)\le (1+C_0|\theta|)^m\wt{Q}(\xi)$ for all $\xi,\theta\in\R^n$, we get
\[
C\wt{Q}(\xi+\eta)\le C(1+C_0|\eta|)^m\wt{Q}(\xi)\le C_1\wt{Q}(\xi),
\]
where the constant $C_1$ does not depend on $Q$. Hence
\[
|Q(\xi+\eta)|\ge \frac{\wt{Q}(\xi)}{C}\ge \frac{\wt{Q}(\xi+\eta)}{C_1}.
\]
Taylor's formula gives $Q(\xi+\eta+\theta)=Q(\xi+\eta)+\sum_{\alpha\neq 0}\frac{Q^{(\alpha)}(\xi+\eta)}{\alpha!}\theta^\alpha$ and then, for $\eta$ chosen as above,
\begin{multline*}
|Q(\xi+\eta+\theta)|\ge |Q(\xi+\eta)|-\sum_{\alpha\neq 0}\frac{|Q^{(\alpha)}(\xi+\eta)|}{\alpha!}|\theta^\alpha|\ge \frac{\wt{Q}(\xi+\eta)}{C_1}-\wt{Q}(\xi+\eta)\sum_{\alpha\neq 0}\frac{1}{\alpha!}|\theta^\alpha|\\
\ge \frac{\wt{Q}(\xi+\eta)}{C_1}-\wt{Q}(\xi+\eta)C_2|\theta|=\wt{Q}(\xi+\eta)(\frac{1}{C_1}-C_2|\theta|)
\end{multline*}
for all $|\theta|\le 1$, with $C_2$ independent of $Q$. It follows that 
\[
|Q(\xi+\eta+\theta)|\ge\frac{1}{2C_1}\wt{Q}(\xi+\eta)
\]
for all $\xi\in\R^n$, for $|\eta|\le 1$ depending on $Q$ and for all $\theta$ with $|\theta|\le (2C_1C_2)^{-1}$. Writing $P$ as $PQ/Q$ we obtain $|P(\xi+\eta+\theta)|\le 2C_1|PQ(\xi+\eta+\theta)|/\wt{Q}(\xi+\eta)$. Concluding, from the first assertion, the property \eqref{est_Hoer} of the polynomial weight functions and the bound from below \eqref{est_dom}, we have, for some $\theta$ suitably smaller than $\min(1,(2C_1C_2)^{-1})$ and $\eta$ depending on $Q$, the inequality
\[
\wt{P}(\xi)\le (1+C_0|\eta|)^m\wt{P}(\xi+\eta)\le C_3|P(\xi+\eta+\theta)|\le C_4\frac{|PQ(\xi+\eta+\theta)|}{\wt{Q}(\xi+\eta)}\le C_5\frac{\wt{PQ}(\xi)}{\wt{Q}(\xi)}
\]
where the constants involved depend only on the order of the polynomials $P$ and $Q$ and the dimension $n$.
\end{proof}
Proposition \ref{prop_dom} clearly holds for representatives $(P_\eps)_\eps$ and $(Q_\eps)_\eps$ of generalized polynomials with the constants $C$, $C'$ and $C''$ independent of $\eps$.
\begin{proposition}
\label{prop_basic_1}
Let $P(D), P_1(D), P_2(D), Q(D), Q_1(D)$ and $Q_2(D)$ be differential operators with constant Colombeau coefficients. 
\begin{itemize}
\item[(i)] If $Q_1(D)\prec P(D)$ and $Q_2(D)\prec P(D)$ then $a_1 Q_1(D)+a_2 Q_2(D)\prec P(D)$ for all $a_1,a_2\in\wt{\C}$. 
\item[(ii)] If $Q_1(D)\prec P_1(D)$ and $Q_2(D)\prec P_2(D)$ then $Q_1Q_2(D)\prec P_1P_2(D)$.
\item[(iii)] $P(D)+aQ(D)\prec P(D)$ for all $a\in\wt{\C}$ if and only if $Q(D)\prec P(D)$.
\end{itemize}
\end{proposition}
\begin{proof}
$(i)$ The first assertion is trivial.\\
$(ii)$ Working at the level of representatives from Proposition \ref{prop_dom}$(ii)$ we can write $\wt{Q_{1,\eps}Q_{2,\eps}}\le C''\wt{Q_{1,\eps}}\wt{Q_{2,\eps}}$. It follows
\[
\wt{Q_{1,\eps}Q_{2,\eps}}\le C''\lambda_{1,\eps}\lambda_{2,\eps}\wt{P_{1,\eps}}\wt{P_{2,\eps}}\le \frac{C''}{C'}\lambda_{1,\eps}\lambda_{2,\eps}\wt{P_{1,\eps}P_{2,\eps}},
\]
with $(\frac{C''}{C'}\lambda_{1,\eps}\lambda_{2,\eps})_\eps$ moderate net.\\
$(iii)$ One direction is clear. Indeed, since $P\prec P$ from the first assertion of this proposition we have that if
$Q\prec P$ then $P+aQ\prec P$ for all $a\in\wt{\C}$. Conversely, let $P+aQ\prec P$. From $-P\prec P$ and $(i)$ we have that $aQ\prec P$. Finally, choosing $a=1$ we obtain $Q\prec P$.
\end{proof}
The following necessary and sufficient conditions for $Q\prec P$ are directly obtained from the first assertion of Proposition \ref{prop_dom}.
\begin{proposition}
\label{prop_basic_2}
Let $Q(D)$ and $P(D)$ be differential operators with constant Colombeau coefficients. The following statements are equivalent:
\begin{itemize}
\item[(i)] $Q(D)\prec P(D)$;
\item[(ii)] there exist representatives $(Q_\eps)_\eps$ and $(P_\eps)_\eps$ and a moderate net $(\lambda_\eps)_\eps$ such that 
\[
|Q_\eps(\xi)|\le \lambda_\eps \wt{P_\eps}(\xi)
\]
for all $\xi\in\R^n$ and $\eps\in(0,1]$;
\item[(iii)] there exist representatives $(Q_\eps)_\eps$ and $(P_\eps)_\eps$ and a moderate net $(\lambda'_\eps)_\eps$ such that 
\[
\wt{Q_\eps}(\xi,t)\le \lambda'_\eps \wt{P_\eps}(\xi,t)
\]
for all $\xi\in\R^n$, for all $\eps\in(0,1]$ and for all $t\ge 1$.
\end{itemize}
\end{proposition}
\begin{proof}
The implications $(i)\Rightarrow(ii)$ and $(iii)\Rightarrow(i)$ are trivial. We only have to prove that $(ii)$ implies $(iii)$. Proposition \ref{prop_dom}$(i)$ yields, for $t\ge 1$,
\begin{multline*}
\wt{Q_\eps}(\xi,t)\le C\sup_{|\eta|<t}|Q_\eps(\xi+\eta)|\le C\lambda_\eps\sup_{|\eta|<t}\wt{P_\eps}(\xi+\eta)=C\lambda_\eps\sup_{|\eta|<t}\wt{P_\eps}(\xi+\eta,1)\le C^2\lambda_\eps\sup_{|\eta|<t+1}\wt{P_\eps}(\xi+\eta)\\
\le C^3\lambda_\eps \wt{P_\eps}(\xi,t+1)\le C^3\lambda_\eps (1+t^{-1})^m \wt{P_\eps}(\xi,t).
\end{multline*}
Hence, $\wt{Q_\eps}(\xi,t)\le \lambda'_\eps\wt{P_\eps}(\xi,t)$ for all $\xi\in\R^n$, $t\ge 1$ and $\eps\in(0,1]$ with $\lambda'_\eps=C^3\lambda_\eps 2^m$.
\end{proof}
The $\G$-elliptic polynomials (see \cite[Section 6]{Garetto:08c}) and their corresponding differential operators can be characterized by means of the order relation $\prec$. We recall that a polynomial $P(\xi)$ with coefficients in $\wt{\C}$ is $\G$-elliptic (or equivalently the operator $P(D)$ is $\G$-elliptic) if there exists a representative $(P_{m,\eps})_\eps$ of $P_m$, a constant $c>0$ and $a\in\R$ such that 
\beq
\label{est_ellip}
|P_{m,\eps}(\xi)|\ge c\eps^a|\xi|^m
\eeq
for all $\eps\in(0,1]$ and for $\xi\in\R^n$. Estimate \eqref{est_ellip} is valid for any representative of $P$ with some other constant $c>0$ and on a smaller interval $(0,\eps_0]$. Due to the homogeneity of $P_{m,\eps}$ it is not restrictive to assume \eqref{est_ellip} valid only for all $\xi$ with $|\xi|=1$.
\begin{proposition}
\label{prop_g_ellip}
Let $P(D)$ be a differential operator of order $m$ with coefficients in $\wt{\C}$. $P(D)$ is stronger than any differential operator with coefficients in $\wt{\C}$ of order $\le m$ if and only if it is $\G$-elliptic.
\end{proposition}
\begin{proof}
We assume that $P(D)$ is $\G$-elliptic and we prove that the $\G$-ellipticity is a sufficient condition. Let $(P_{m,\eps})_\eps$ be a representative of $P_m$ such that $|P_{m,\eps}(\xi)|\ge c\eps^a|\xi|^m$ for some constants $c>0$ and $a\in\R$, for all $\xi\in\R^n$ and for all $\eps\in(0,1]$. It follows that $P_m(\xi)$ is invertible in any $\xi=\xi_0$ of $\R^n$ and therefore from Remark \ref{rem_HO_inv} (and more precisely from \cite[Proposition 7.6]{HO:03}) we have that $\wt{P}$ is invertible in $\xi_0$. 
The estimate $|P_{m,\eps}(\xi)|\ge c\eps^a|\xi|^m$ yields
\[
c\eps^a|\xi|^m\le |P_{m,\eps}(\xi)|\le |P_\eps(\xi)|+|P_\eps(\xi)-P_{m,\eps}(\xi)|\le |P_\eps(\xi)|+C_\eps(1+|\xi|^{m-1}),
\]
where $(c_\eps)_\eps$ is a strictly nonzero net. Assuming $|\xi|\ge 2C_\eps c^{-1}\eps^{-a}$ we obtain the inequality
\[
c\eps^a |\xi|^m\le 2|P_\eps(\xi)|+2C_\eps\le 2\wt{P_\eps}(\xi)+2C_\eps.
\]
For $|\xi|\ge R_\eps$, where $R_\eps=\max\{(4C_\eps c^{-1}\eps^{-a}+1)^{\frac{1}{m}},2C_\eps c^{-1}\eps^{-a}\}$, the following bound from below
\beq
\label{est_Hoer_1}
\frac{c}{4}\eps^{a}(1+|\xi|^m)\le \wt{P_\eps}(\xi)
\eeq
holds. Since, from the invertibility of $\wt{P}$ in $\xi_0$ there exists a strictly nonzero net $(\lambda_\eps)_\eps$ such that
\[
\wt{P_\eps}(\xi)\ge \lambda_\eps(1+C|\xi_0-\xi|)^{-m},
\]
for all $\eps\in(0,1]$, we can extend \eqref{est_Hoer_1} to all $\xi\in\R^n$. More precisely,
\[
\lambda_\eps(1+C|\xi_0|+CR_\eps)^{-m}(1+R_\eps)^{-m}(1+|\xi|)^m\le \lambda_\eps (1+C|\xi_0-\xi|)^{-m}(1+|\xi|)^{-m}(1+|\xi|)^m\le \wt{P_\eps}(\xi)
\]
holds for $|\xi|\le R_\eps$. The net $(R_\eps)_\eps$ is strictly nonzero. Hence there exists a moderate net $\omega_\eps$ such that
\[
(1+|\xi|^m)\le\omega_\eps \wt{P_\eps}(\xi)
\]
for all $\eps\in(0,1]$ and $\xi\in\R^n$. Now if $Q(D)$ is a differential operator with coefficients in $\wt{\C}$ of order $m'\le m$, \eqref{est_Hoer} yields
\[
\wt{Q_\eps}(\xi)\le \wt{Q_\eps}(0)(1+C|\xi|)^{m'}.
\]
Hence,
\[
\wt{Q_\eps}(\xi)\le \wt{Q_\eps}(0)c'(1+|\xi|^m)\le c'\wt{Q_\eps}(0)\omega_\eps\,\wt{P_\eps}(\xi),
\]
where $(c'\wt{Q_\eps}(0)\omega_\eps)_\eps$ is moderate. This means that $Q(D)\prec P(D)$.

We now prove that the $\G$-ellipticity of $P(D)$ is necessary in order to have $Q(D)\prec P(D)$ for all $Q$ of order $\le m$. If $P(D)$ is not $\G$-elliptic then we can find a representative $(P_\eps)_\eps$, a decreasing sequence $\eps_q\to 0$ and a sequence $\xi_{\eps_q}$ with $|\xi_{\eps_q}|=1$ such that
\[
|P_{m,\eps_q}(\xi_{\eps_q})|<\eps_q^q
\]
for all $q\in\N$. We set $\xi_\eps=\xi_{\eps_q}$ for $\eps=\eps_q$ and $0$ otherwise. By construction $(|\xi_\eps|)_\eps\not\in\Neg$ and 
\[
(P_{m,\eps}(\xi_\eps))_\eps\in\Neg.
\]
Since there exists a moderate net $(c_\eps)_\eps$ such that
\[
\wt{P_\eps}^2(t\xi)\le 2t^{2m}|P_{m,\eps}(\xi)|^2+c_\eps t^{2m-2}\lara{\xi}^{2m-2}
\]
for all $\xi\in\R^n$ and $t\ge 1$, we obtain
\[
\wt{P_\eps}(t\xi_\eps)\le t^m n_\eps+c'_\eps t^{m-1},
\]
where $(n_\eps)_\eps\in\Neg$ and $(c'_\eps)_\eps\in\EM$. Since $(|\xi_\eps|)_\eps\in\EM\setminus \Neg$ there exists a component $\xi_{i,\eps}$ such that $(|\xi_{i,\eps}|)_\eps\in\EM\setminus\Neg$. We take the homogeneous polynomial of degree $m$ $Q(\xi)=\xi_i^m$ and we prove that $Q(D)\not\prec P(D)$. By Proposition \ref{prop_basic_2} assume that there exists representatives $(Q_\eps)_\eps$ and $(P'_\eps)_\eps$ of $Q$ and $P$ respectively and a moderate net $(\lambda_\eps)_\eps$ such that
\[
|Q_\eps(\xi)|\le \lambda_\eps\wt{P'_\eps}(\xi)
\]
for all $\xi$ and $\eps$. $Q_\eps(\xi)$ is of the form $(1+n_{1,\eps})\xi_i^m$ where $(n_{1,\eps})_\eps\in\Neg$ and concerning $\wt{P'_\eps}(\xi)$ we have $\wt{P'_\eps}^2(\xi)\le 2\wt{P_\eps}^2(\xi)+n_{2,\eps}\lara{\xi}^{2m}$ with $(n_{2,\eps})_\eps\in\Neg$. This entails the estimate
\[
|(1+n_{1,\eps})\xi_i^m|\le \lambda'_\eps\wt{P_\eps}(\xi)+n_{3,\eps}\lara{\xi}^m
\]
and for $\xi= t\xi_\eps$, $t\ge 1$,
\[
|1+n_{1,\eps}|t^m|\xi_{i,\eps}^m|\le \lambda'_\eps\wt{P_\eps}(t\xi_\eps)+n_{3,\eps}\lara{t\xi_\eps}^m\le n_{4,\eps}t^m+c''_\eps t^{m-1},
\]
where $(c''_\eps)_\eps$ is moderate and strictly nonzero. The inequality
\[
\frac{|1+n_{1,\eps}||\xi_{i,\eps}^m|-n_{4,\eps}}{c''_\eps}\le t^{-1}
\]
is valid for all $t\ge 1$. Hence $|1+n_{1,\eps}||\xi_{i,\eps}^m|-n_{4,\eps}\le 0$ and from the invertibility of $(|1+n_{1,\eps}|)_\eps$ we get
\[
|\xi_{i,\eps}|^m\le \frac{n_{4,\eps}}{|1+n_{1,\eps}|},
\]
with $(\frac{n_{4,\eps}}{|1+n_{1,\eps}|})_\eps\in\Neg$. Concluding the net $(|\xi_{i,\eps}|)_\eps$ is negligible in contradiction with our assumptions.
\end{proof}
We introduce another order relation which is closely connected with $Q(D)\prec P(D)$.
\begin{definition}
\label{def_domin}
Let $P(D)$ and $Q(D)$ be differential operators with coefficients in $\wt{\C}$. We say that $P(D)$ dominates $Q(D)$ (and we write $P(D)\succ\succ Q(D)$ or $Q(D)\prec\prec P(D)$) if there exist 
\begin{itemize}
\item[-] representatives $(P_\eps)_\eps$ and $(Q_\eps)_\eps$ of $P$ and $Q$ respectively, 
\item[-] a moderate net $(\lambda_\eps)_\eps$,
\item[-] a function $C(t)>0$ with $\lim_{t\to +\infty}C(t)=0$ and the property
\[
\forall a\in\R\ \exists b\in\R\ \forall t\ge \eps^b\qquad C(t)\le \eps^a,
\]
\end{itemize}
such that
\[
\wt{Q_\eps}(\xi,t)\le \lambda_\eps C(t)\wt{P_\eps}(\xi,t)
\]
for all $\xi\in\R^n$, $\eps\in(0,1]$ and $t\ge 1$.
\end{definition}
Clearly $Q(D)\prec\prec P(D)$ implies $Q(D)\prec P(D)$ and $P(D)$ dominates $P^{(\alpha)}(D)$ for all $\alpha\neq 0$. Indeed, $\wt{P^{(\alpha)}_\eps}(\xi,t)\le t^{-|\alpha|}\wt{P_\eps}(\xi,t)$. 
\begin{proposition}
\label{prop_basic_3}
Let $P(D), P_1(D), P_2(D), Q_1(D)$ and $Q_2(D)$ be differential operators with constant Colombeau coefficients. 
\begin{itemize}
\item[(i)] If $Q_1(D)\prec\prec P(D)$ and $Q_2(D)\prec\prec P(D)$ then $a_1 Q_1(D)+a_2 Q_2(D)\prec\prec P(D)$ for all $a_1,a_2\in\wt{\C}$. 
\item[(ii)] If $Q_1(D)\prec\prec P_1(D)$ and $Q_2(D)\prec P_2(D)$ then $Q_1Q_2(D)\prec\prec P_1P_2(D)$.
\end{itemize}
\end{proposition}
\begin{proof}
The first statement is trivial. By applying Proposition \ref{prop_dom}$(ii)$ to $P(t\xi)$ and $Q(t\xi)$ we obtain for any polynomials $P$ and $Q$ the estimate
\beq
\label{est_dez}
C'\wt{P}(\xi,t)\wt{Q}(\xi,t)\le \wt{PQ}(\xi,t)\le C''\wt{P}(\xi,t)\wt{Q}(\xi,t)
\eeq
where $C'$ and $C''$ depend only on the order of $P$ and $Q$. Hence, for all $t\ge 1$ we have
\[
\wt{Q_{1,\eps}Q_{2,\eps}}(\xi,t)\le C''\wt{Q_{1,\eps}}(\xi,t)\wt{Q_{2,\eps}}(\xi,t)\le C''\lambda_{1,\eps} C(t)\wt{P_{1,\eps}}(\xi,t)\wt{Q_{2,\eps}}(\xi,t).
\]
Proposition \ref{prop_basic_2}$(iii)$ combined with the estimate \eqref{est_dez} yields
\[
\wt{Q_{1,\eps}Q_{2,\eps}}(\xi,t)\le C''\lambda_{1,\eps} C(t)\wt{P_{1,\eps}}(\xi,t)\lambda_{2,\eps}\wt{P_{2,\eps}}(\xi,t)\le \lambda_\eps C(t)\wt{P_{1,\eps}P_{2,\eps}}(\xi,t),
\]
valid for all $t\ge 1$, for all $\eps\in(0,1]$ and for all $\xi\in\R^n$.
\end{proof}
The order relation $\prec\prec$ is used in comparing an operator of principal type with a differential operator of order strictly smaller.
\begin{definition}
\label{def_princ_point}
A partial differential operator $P(D)$ with constant Colombeau coefficients is said to be of principal type if there exists a representative $(P_{m,\eps})_\eps$ of the principal symbol $P_{m}$, $a\in\R$ and $c>0$ such that
\[
|\nabla_\xi P_{m,\eps}(\xi)|\ge c\eps^a|\xi|^{m-1}
\]
for all $\eps\in(0,1]$ and all $\xi\in\R^n$.
\end{definition}
As for $\G$-elliptic operators the previous estimate holds for any representative $(P_\eps)_\eps$ of $P$, for some constant $c$ and in an enough small interval $(0,\eps_0]$.
\begin{proposition}
\label{prop_princ_type}
Let $P(D)$ be a differential operator with coefficients in $\wt{\C}$ of principal type and degree $m$ and let one of the coefficients of $P_m(D)$ be invertible. Hence,
\begin{itemize}
\item[(i)] $P(D)$ dominates any differential operator with coefficients in $\wt{\C}$ of order $\le m-1$;
\item[(ii)] if $Q(D)$ has order $m$ and there exists a moderate net $(\lambda_\eps)_\eps$ and representatives $(Q_{m,\eps})_\eps$ and $(P_{m,\eps})_\eps$ such that
\[
|Q_{m,\eps}(\xi)|\le \lambda_\eps |P_{m,\eps}(\xi)|
\]
for all $\xi\in\R^n$ and $\eps\in(0,1]$, then $Q(D)\prec P(D)$.
\end{itemize}
\end{proposition}
\begin{proof}
$(i)$ Let $(P_{m,\eps})_\eps$ be a representative of $P_m$ such that $|\nabla P_{m,\eps}(\xi)|\ge c\eps^a|\xi|^{m-1}$ for some constants $c>0$ and $a\in\R$, for all $\xi\in\R^n$ and for all $\eps\in(0,1]$. We have, for some strictly nonzero net $(C_\eps)_\eps$ the inequality
\[
|\nabla_\xi P_\eps(\xi)|\ge |\nabla_\xi P_{m,\eps}(\xi)|-|\nabla_\xi(P_\eps-P_{m,\eps})|\ge c'\eps^a\lara{\xi}^{m-1}-C_\eps\lara{\xi}^{m-2}\ge \frac{1}{2}c'\eps^a(1+|\xi|^{m-1}),
\]
valid for $|\xi|\ge R_\eps$ with $(R_\eps)_\eps$ moderate and big enough. Hence,
\[
\biggl(\sum_{\alpha\neq 0}|P^{(\alpha)}_\eps(\xi)|^2\biggr)^{\frac{1}{2}}\ge |\nabla_\xi P_\eps(\xi)|\ge \frac{1}{2}c'\eps^a(1+|\xi|^{m-1})
\]
for $|\xi|\ge R_\eps$. 
From the invertibility of one of the coefficients of the principal part we get the bound from below
\[
\omega_\eps\le \sum_{\alpha\neq 0}|P^{(\alpha)}_\eps(\xi)|^2,
\]
where $(\omega_\eps)_\eps$ is moderate. It follows, for $|\xi|\le R_\eps$, 
\[
\sum_{\alpha\neq 0}|P^{(\alpha)}_\eps(\xi)|^2\ge \omega_\eps(1+|\xi|^{m-1})^{-2}(1+|\xi|^{m-1})^2\ge \omega_\eps ((1+(R_\eps)^{m-1})^{-2}(1+|\xi|^{m-1})^2.
\]
Summarizing, we find a moderate and strictly nonzero net $(\lambda_\eps)_\eps$ such that 
\[
\lambda_\eps^2(1+|\xi|^{2m-2})\le \sum_{\alpha\neq 0}|P^{(\alpha)}_\eps(\xi)|^2
\]
for all $\xi\in\R^n$ and $\eps\in(0,1]$. This implies 
\beq
\label{est_oggi}
\lambda_\eps^2 t^2(1+|\xi|^{2m-2})\le \sum_{\alpha\neq 0}t^{2|\alpha|}|P^{(\alpha)}_\eps(\xi)|^2\le \wt{P_\eps}(\xi,t),
\eeq
for all $t\ge 1$, $\xi\in\R^n$ and $\eps\in(0,1]$.

Let $Q(D)$ be a differential operator with coefficients in $\wt{\C}$ and order $m'\le m-1$. We have, for some moderate net $(c_\eps)_\eps$ the inequality
\[
|Q_\eps(\xi)|\le c_\eps(1+|\xi|)^{m-1}
\]
and therefore from \eqref{est_oggi}
\[
|Q_\eps(\xi)|\le (c_\eps \lambda_\eps^{-1}t^{-1})t^{1}\lambda_\eps(1+|\xi|^{m-1})\le (c'_\eps \lambda_\eps^{-1}t^{-1})\wt{P_\eps}(\xi,t).
\]
Arguing as in proof of Proposition \ref{prop_basic_2} and making use of Proposition \ref{prop_dom} we obtain that there exists a moderate net $(c''_\eps)_\eps$ such that
\[
\wt{Q_\eps}(\xi,t)\le c''_\eps t^{-1}\wt{P_\eps}(\xi,t).
\]
Indeed
\begin{multline*}
\wt{Q_\eps}(\xi,t)\le C\sup_{|\eta|<t}|Q_\eps(\xi+\eta)|\le Cc'_\eps\lambda^{-1}_\eps t^{-1}\sup_{|\eta|<t}\wt{P_\eps}(\xi+\eta,t)\le C^2c'_\eps\lambda^{-1}_\eps t^{-1}\sup_{|\eta|<t}\sup_{|\theta|<t}|P_\eps(\xi+\eta+\theta)|\\
\le C^2c'_\eps\lambda^{-1}_\eps t^{-1}\sup_{|\eta|<2t}|P_\eps(\xi+\eta)|\le C^3c'_\eps\lambda^{-1}_\eps t^{-1}\wt{P_\eps}(\xi,2t)\le C^3 2^m c'_\eps\lambda^{-1}_\eps t^{-1}\wt{P_\eps}(\xi,t).
\end{multline*}
This means that $Q(D)\prec\prec P(D)$.

$(ii)$ Let $Q(D)$ be a differential operator of order $m$ satisfying the condition $(ii)$. From the first statement we already know that $Q(D)-Q_{m}(D)\prec\prec P(D)$ and thus $Q(D)-Q_{m}(D)\prec P(D)$. It remains to prove that $P(D)$ is stronger than $Q_m(D)$. Writing $\wt{Q_{m,\eps}}^2$ as $|Q_{m,\eps}|^2+\sum_{\beta\neq 0}|Q^{(\beta)}_{m,\eps}|^2$, and since the second term has order $\le m-1$, we have
\begin{multline*}
\wt{Q_{m,\eps}}^2(\xi)\le \lambda_\eps^2|P_{m,\eps}(\xi)|^2+\lambda_{1,\eps}\wt{P_{\eps}}^2(\xi)\le 2\lambda_\eps^2|P_{\eps}(\xi)|^2+2\lambda_\eps^2|(P_{\eps}-P_{m,\eps})(\xi)|^2+\lambda_{1,\eps}\wt{P_{\eps}}^2(\xi)\\
\le(2\lambda_\eps^2+\lambda_{2,\eps}+\lambda_{1,\eps})\wt{P_\eps}^2(\xi).
\end{multline*}
Hence $Q_m(D)\prec P(D)$.
\end{proof}
The following proposition determines a family of equally strong operators.
\begin{proposition}
\label{prop_eq_strong}
Let $P(D)$ and $Q(D)$ be differential operators with coefficients in $\wt{\C}$. If $Q(D)\prec\prec P(D)$ then $P(D)\prec P(D)+aQ(D)\prec P(D)$ for all $a\in\wt{\C}$.
\end{proposition}
\begin{proof}
Since $Q(D)\prec\prec P(D)$ implies $Q(D)\prec P(D)$, from the third assertion of Proposition \ref{prop_basic_1} we have that $P(D)+aQ(D)\prec P(D)$ for all $a\in\wt{\C}$. We now fix $a\in\wt{\C}$ and take $R(D)=P(D)+aQ(D)$. Arguing at the level of representatives we obtain
\[
\wt{P_\eps}^2(\xi,t)=\sum_\alpha |R_\eps^{(\alpha)}-a_\eps Q_\eps^{(\alpha)}|^2(\xi)t^{2|\alpha|}\le 2\wt{R_\eps}^2(\xi,t)+2|a_\eps|^2\wt{Q_\eps}^2(\xi,t)\le 2\wt{R_\eps}^2(\xi,t)+2|a_\eps|^2\lambda_\eps^2 C^2(t)\wt{P_\eps}^2(\xi,t).
\]
By the moderateness assumption we have that $|a_\eps|^2\lambda_\eps^2\le \eps^{2a}$ for all $\eps$ small enough. Choosing $b\in\R$ such that $C(t)\le \eps^{-a+1}$ for $t\ge\eps^b$ we can write the inequality
\[
\wt{P_\eps}^2(\xi,t)\le  2\wt{R_\eps}^2(\xi,t)+2\eps^{2a}\eps^{-2a+2}\wt{P_\eps}^2(\xi,t)
\]
for $t\ge\eps^b$, for $\eps$ small enough and for all $\xi$. It follows that 
\[
\wt{P_\eps}^2(\xi,t)\le  2\wt{R_\eps}^2(\xi,t)+\frac{1}{2}\wt{P_\eps}^2(\xi,t)
\]
for $t\ge\max(1,\eps^b)$ and $\eps\in(0,\eps_0]$ with $\eps_0\le 2^{-1}$. Hence,
\[
\wt{P_\eps}(\xi,t)\le 2\wt{R_\eps}(\xi,t),
\] 
under the same conditions on $\eps$ and $t$. Let $(t_\eps)_\eps\in\EM$ with $t_\eps\ge \max(1,\eps^b)$. We can write
\[
\wt{P_\eps}(\xi)\le \wt{P_\eps}(\xi,t_\eps)\le 2\wt{R_\eps}(\xi,t_\eps)\le 2t_\eps^m\wt{R_\eps}(\xi)\le \lambda_\eps\wt{R_\eps}(\xi),
\]
valid for some moderate net $(\lambda_\eps)_\eps$ and for $\eps\in(0,\eps_0]$. This means that $P(D)\prec R(D)=P(D)+aQ(D)$.
\end{proof}
The next corollary is straightforward from $P^{(\alpha)}(D)\prec\prec P(D)$.
\begin{corollary}
\label{corol_eq_strong}
For all $\alpha\in\N^n$,
\[
P(D)\prec P(D)+P^{(\alpha)}(D)\prec P(D).
\]
\end{corollary}

\section{Parametrices and local solvability}
\label{sec_parametrix}
We begin our investigation of locally solvable differential operators in the Colombeau framework, by showing that differential operators which admits a generalized pseudodifferential parametrix (at least a right generalized parametrix) are locally solvable. Some needed notions of generalized pseudodifferential operator theory are collected in the following subsection.
\subsection{Preliminary notions of generalized pseudodifferential operator theory}
\paragraph{Symbols.}
Throughout the paper $S^m(\R^{2n})$ denotes the space of H\"ormander symbols fulfilling global estimates on $\R^{2n}$. In detail $|a|^{(m)}_{\alpha,\beta}$ is the seminorm
\[
\sup_{(x,\xi)\in\R^{2n}}\lara{\xi}^{-m+|\alpha|}|\partial^\alpha_\xi\partial^\beta_x a(x,\xi)|.
\]
In a local context, that is on an open subset $\Om$ of $\R^n$, we work with symbols that satisfy uniform estimates on compact subsets of $\Om$. In this case we use the notation $S^{m}(\Om\times\R^n)$ and the seminorms
\[
|a|^{(m)}_{K,\alpha,\beta}:=\sup_{x\in K\Subset\Om, \xi\in\R^n}\lara{\xi}^{-m+|\alpha|}|\partial^\alpha_\xi\partial^\beta_x a(x,\xi)|.
\]
The corresponding sets of generalized symbols are introduced by means of the abstract models $\G_E$ and $\G^\ssc_E$ introduced in Section \ref{sec_basic} where $E=S^m(\R^{2n})$ or $E=S^m(\Om\times\R^n)$. 

\paragraph{Mapping properties.}
A theory of generalized pseudodifferential operators has been developed in \cite{Garetto:04, GGO:03} for symbols in $\G_{{{S}}^m(\R^{2n})}$ and $\G^{\ssc}_{{{S}}^m(\R^{2n})}$ and for more elaborated notions of generalized symbols and amplitudes. We address the reader to the basic notions section of \cite{Garetto:ISAAC07} for an elementary introduction to the subject. In the sequel $\GS(\R^n)$ is the Colombeau space based on $E=\S(\R^n)$ and $\GSinf(\R^n)$ is the subspace of $\GS(\R^n)$ of those generalized functions $u$ having a representative $(u_\eps)_\eps$ fulfilling the following condition:
\[
\exists N\in\N\ \forall\alpha,\beta\in\N^n\qquad\qquad \sup_{x\in\R^n}|x^\alpha\partial^\beta u_\eps(x)|=O(\eps^{-N}).
\]
Finally $\LL(\GS(\R^n),\wt{\C})$ is the topological dual of $\GS(\R^n)$.

Let now $p\in\G_{{{S}}^m(\R^{2n})}$. The pseudodifferential operator
\[
p(x,D)u=\int_{\R^n}\esp^{ix\xi}p(x,\xi)\widehat{u}(\xi)\, \dslash\xi
\]
\begin{itemize}
\item[(i)] maps $\GS(\R^n)$ into $\GS(\R^n)$,
\item[(ii)] can be continuously extended to a $\wt{\C}$-linear map on $\LL(\GS(\R^n),\wt{\C})$,
\item[(iii)] maps basic functionals into basic functionals,
\item[(iv)] maps $\GSinf(\R^n)$ into itself if $p$ is of slow scale type.
\end{itemize}
\paragraph{Generalized symbols and asymptotic expansions}
The notion of asymptotic expansion for generalized symbols in $\G_{S^m(\R^{2n})}$ is based on the following definition at the level of representatives. 
\begin{definition}
\label{def_asymp}
Let $\{m_j\}_{j\in\mathbb{N}}$ be sequences of real numbers with
$m_j\searrow -\infty$, $m_0=m$. Let $\{(a_{j,\epsilon})_\epsilon\}_{j\in\mathbb{N}}$
be a sequence of elements $(a_{j,\epsilon})_\epsilon\in\mM_{S^{m_j}(\R^{2n})}$. 
We say that the formal series $\sum_{j=0}^\infty(a_{j,\epsilon})_\epsilon$ is the asymptotic expansion of
$(a_\epsilon)_\epsilon\in\mathcal{E}[\R^{2n}]$,
$(a_\epsilon)_\epsilon\sim\sum_j(a_{j,\epsilon})_\epsilon$ for short, iff for all $r\ge 1$
\[
\biggl(a_\epsilon-\sum_{j=0}^{r-1}a_{j,\epsilon}\biggr)_\epsilon\in \mM_{S^{m_r}(\R^{2n})}. 
\]
\end{definition}
By arguing as in \cite[Theorem 2.2]{Garetto:ISAAC07} one proves that there exists a net of symbols with a given asymptotic expansion according to Definition \ref{def_asymp}.
\begin{theorem}
\label{theo_asymp}
Let $\{(a_{j,\epsilon})_\epsilon\}_{j\in\mathbb{N}}$
be a sequence of elements $(a_{j,\epsilon})_\epsilon\in\mM_{S^{m_j}(\R^{2n}}$ with $m_j\searrow -\infty$ and $m_0=m$. Then, there exists $(a_\eps)_\eps\in\mM_{S^{m}(\R^{2n})}$ such that $(a_\epsilon)_\epsilon\sim\sum_j(a_{j,\epsilon})_\epsilon$. Moreover, if $(a'_\epsilon)_\epsilon\sim\sum_j(a_{j,\epsilon})_\epsilon$ then $(a_\eps-a'_\eps)_\eps\in\mM_{S^{-\infty}(\R^{2n})}$.
\end{theorem}
We now take in consideration regular nets of symbols. Inspired by the notations of \cite{Garetto:04} we say that $(a_\eps)_\eps$ belongs to $\mM_{S^m(\R^{2n}),b}$ if and only if $|a_\eps|^{(m)}_{\alpha,\beta}=O(\eps^b)$ for all $\alpha$ and $\beta$. In other words we require the same kind of moderateness for all orders of derivatives. A closer look to the proof of Theorem 2.2 in \cite{Garetto:ISAAC07} yields the following corollary.
\begin{corollary}
\label{corol_asymp}
Let $\{(a_{j,\epsilon})_\epsilon\}_{j\in\mathbb{N}}$ as in Theorem \ref{theo_asymp}. If $(a_{j,\eps})_\eps\in\mM_{S^{m_j}(\R^{2n}),b}$ for each $j$ then there exists $(a_\eps)_\eps\in\mM_{S^{m}(\R^{2n}),b}$ such that 
\[
\biggl(a_\epsilon-\sum_{j=0}^{r-1}a_{j,\epsilon}\biggr)_\epsilon\in \mM_{S^{m_r}(\R^{2n}),b}.
\]
for every $r\ge 1$. This result is unique modulo $\mM_{S^{-\infty}(\R^{2n}),b}$.
\end{corollary}
Note that this statement recalls the first concept of asymptotic expansion for nets of symbols studied in \cite{Garetto:04} but avoids global estimates on the $\eps$-interval $(0,1]$.

It is clear that when $p\in\G_{S^m(\R^{2n})}$ has a representative in $\cup_{b\in\R}\mM_{S^m(\R^{2n}),b}$ then $p(x,D)$ maps $\GSinf(\R^n)$ into $\GSinf(\R^n)$.
\paragraph{Kernels and regularizing operators.}
Any generalized pseudodifferential operator has a kernel in $\LL(\GS(\R^{2n}),\wt{\C})$ but when $p\in\G_{{{S}}^{-\infty}(\R^{2n})}$ then $k_p\in\Gt(\R^{2n})$ with a representative $(k_{p,\eps})_\eps$ fulfilling the following property: 
\begin{multline}
\label{prop_kern}
\forall\alpha,\beta\in\N^n\, \forall d\in \N\, \exists (\lambda_\eps)_\eps\in\EM\, \forall\eps\in(0,1]\\ \sup_{(x,y)\in\R^{2n}}\lara{x}^{-d}\lara{y}^d|\partial^{\alpha}_x\partial^{\beta}_y k_{p,\eps}(x,y)|\le\lambda_\eps,\\
\sup_{(x,y)\in\R^{2n}}\lara{x}^{d}\lara{y}^{-d}|\partial^{\alpha}_x\partial^{\beta}_y k_{p,\eps}(x,y)|\le\lambda_\eps.
\end{multline}
If $p\in\G^\ssc_{{{S}}^{-\infty}(\R^{2n})}$ then $(\lambda_\eps)_\eps$ in \eqref{prop_kern} is a slow scale net. With a symbol of order $-\infty$ the pseudodifferential operator $p(x,D)$ can be written in the form 
\[
p(x,D)u=\int_{\R^n}k_p(x,y)u(y)\, dy.
\]
It maps $\LL(\GS(\R^n),\wt{\C})$ into $\Gt(\R^n)$ and $\LL(\G(\R^n),\wt{\C})$ into $\GS(\R^n)$. If $p$ is of slow scale type then the previous mappings have image in $\Gtinf(\R^n)$ and $\GSinf(\R^n)$ respectively.

\paragraph{$L^2$-continuity.}
We finally discuss some $L^2$-continuity. From the well-known estimate (see \cite[Chapter2, Theorem 4.1]{Kumano-go:81})
\[
\Vert a(x,D)u\Vert_2\le C_0\max_{|\alpha+\beta|\le l_0}|a|^{(0)}_{\alpha,\beta}\Vert u\Vert_2,\qquad\qquad\qquad \text{for}\ u\in\S(\R^n)
\]
valid for $a\in S^0(\R^n)$, for some $l_0>0$ and for a constant $C_0$ depending on the space dimension $n$, one easily has that a generalized pseudodifferential operator $p(x,D)$ with symbol $p\in\G_{{{S}}^0(\R^{2n})}$ maps $\G_{L^2(\R^n)}$ continuously into itself. If we now consider a basic functional $T$ of $\LL(\GS(\R^n),\wt{\C})$ given by a net $(T_\eps)_\eps\in\mM_{L^2(\R^n)}$, we have that $p(x,D)T$ is a basic functional in $\LL(\GS(\R^n),\wt{\C})$ with the same ${L^2}$-structure. We introduce the notation $\LL_2(\GS(\R^n),\wt{\C})$ for the set of basic functionals in $\LL(\GS(\R^n),\wt{\C})$ with a representative in $\mM_{L^2(\R^n)}$. Hence, a pseudodifferential operator with symbol in $p\in\G_{{{S}}^0(\R^{2n})}$ has the mapping property
\beq
\label{L_2_cont}
p(x,D):\LL_2(\GS(\R^n),\wt{\C})\to \LL_2(\GS(\R^n),\wt{\C}).
\eeq
Analogously, in the dual $\LL(\Gc(\R^n),\wt{\C})$ one can define the subset $\LL_{2,{\rm{loc}}}(\Gc(\R^n),\wt{\C})$ of those basic functionals $T$ defined by a net $(T_\eps)_\eps\in\mM_{L^2_{\rm{loc}}(\R^n)}$, i.e. $(\phi T_\eps)_\eps\in\mM_{L^2(\R^n)}$ for all $\phi\in\Cinfc(\R^n)$.

If $P(x,D)$ is a differential operator with Colombeau coefficients such that $P(x,\xi)\in\G_{S^{m}(\R^{2n})}$ then in addition to the mapping properties as a pseudodifferential operator we have that the restriction to any open subset $\Om$ maps $\Gc(\Om)$, $\G(\Om)$, $\LL(\G(\Om),\wt{\C})$ and $\LL(\Gc(\Om),\wt{\C})$ into themselves respectively. Typical example is obtained by taking the coefficients $c_\alpha$ of $P(x,D)=\sum_{|\alpha|\le m}c_\alpha(x)D^\alpha$ in the algebra $\G_E$ with $E={\cap_s W^{s,\infty}(\R^n)}$. In this case one can use the notation $\G_\infty(\R^n)$ for simplicity.
\subsection{A first sufficient condition of local solvability}
\begin{theorem}
\label{theo_sol_par}
Let $P(x,D)=\sum_{|\alpha|\le m}c_\alpha(x)D^\alpha$ be a differential operator with coefficients $c_\alpha\in\G_\infty(\R^n)$. Let $(P_\eps)_\eps$ a representative of $P$. If 
\begin{itemize}
\item[(i)] there exist $(q_\eps)_\eps\in\mM_{S^{m'}(\R^{2n})}$ with $m'\le 0$ and $(r)_\eps\in\mM_{S^{-\infty}(\R^{2n})}$ such that
\[
P_\eps(x,D)q_\eps(x,D)=I+r_\eps(x,D)
\]
on $\S(\R^n)$ for all $\eps\in(0,1]$,
\item[(ii)] there exists $l<-n$ such that
\[
|r_\eps|^{(l)}_{0,0}=O(1),
\]
\end{itemize}
then for all $x_0\in\R^n$ there exists a neighborhood $\Om$ of $x_0$ and a cut-off function $\phi$, identically $1$ near $x_0$, such that the following solvability result holds:
\beq
\label{first_solv}
\forall F\in\LL_{2,{\rm{loc}}}(\Gc(\R^n),\wt{\C})\ \exists T\in\LL(\Gc(\Om),\wt{\C})\qquad\quad P(x,D)T=\phi F\qquad \text{on $\Om$}.
\eeq
\end{theorem}
\begin{proof}
We begin by dealing with the regularizing operator $r_\eps$. From $(ii)$ it follows that 
\[
r_\eps(x,D)u=\int_{\R^n}k_{r_\eps}(x,y)u(y)\, dy=\int_{\R^n}\int_{\R^n}\esp^{i(x-y)\xi}r_\eps(x,\xi)\, \dslash\xi\, u(y)\, dy\biggr)_\eps\biggr],\qquad\quad u\in\S(\R^n)
\]
with 
\[
k_{r_\eps}(x,y)=\int_{\R^n}\esp^{i(x-y)\xi}r_\eps(x,\xi)\, \dslash\xi
\]
and
\beq
\label{est_k_eps}
\sup_{x\in \R^n, y\in\R^n} |k_{r_\eps}(x,y)|=O(1).
\eeq
We now take a neighborhood $\Om$ of $x_0$ and a cut-off $\phi\in\Cinfc(\Om)$ and investigate the properties of the net of operators $r_\eps(x,D)\phi$ on $\Om$. For all $g\in L^2(\Om)$ we have that $\phi g\in L^2(\R^n)$ and therefore
\[
r_\eps(x,D)(\phi g)|_\Om=\biggl(\int_\Om k_{r_\eps}(x,y)\phi(y)g(y)\, dy\biggr){\big|_\Om}.
\]
This net of distributions actually belongs to $L^2(\Om)$. Indeed,
\[
\Vert r_\eps(x,D)(\phi g)|_\Om \Vert_{2}\le \int_\Om \biggl(\int_\Om |k_{r_\eps}(x,y)|^2\, dx\biggr)^{\frac{1}{2}}|\phi(y) g(y)|\, dy\le |\Om|\sup_{\Om\times\Om}|k_{r_\eps}(x,y)|\Vert \phi\Vert_2 \Vert g\Vert_2.
\]
From \eqref{est_k_eps} by choosing $\Om$ small enough and a suitable $\phi\in\Cinfc(\Om)$, we obtain that 
\[
\Vert r_\eps(x,D)(\phi g)|_\Om \Vert_{2}\le \frac{1}{2}\Vert g\Vert_2
\]
for all $g\in L^2(\Om)$ uniformly on an interval $(0,\eps_0]$. In other words the net of operators
\[
\wt{r}_{\eps}:L^2(\Om)\to L^2(\Om),\qquad \wt{r}_\eps(g)=r_\eps(x,D)(\phi g)|_\Om
\]
has operator norm $\le\frac{1}{2}$ for all $\eps\in(0,\eps_0]$. In the same way we define
\[
\wt{I}:L^2(\Om)\to L^2(\Om),\qquad \wt{I}(g)=\phi g
\]
and
\[
\wt{q}_\eps:L^2(\Om)\to L^2(\Om),\qquad \wt{q}_\eps(g)=q_\eps(x,D)(\phi g)|_\Om.
\]
This last mapping property follows from $q_\eps(x,D):L^2(\R^n)\to L^2(\R^n)$ valid because $m'\le 0$. One can choose $\phi$ such that $\Vert \wt{I}-I\Vert$ is very small and in particular $\Vert \wt{I}-I+\wt{r}_\eps\Vert <1$ uniformly on $(0,\eps_0]$. The series $\sum_{n=0}^\infty\Vert \wt{I}-I+\wt{r}_\eps\Vert^n$ is convergent. Hence, from Theorem 2 in \cite[Chapter 2]{Yosida:80} we have that 
$\wt{I}+\wt{r}_\eps$ has a continuous linear inverse on $L^2(\Om)$ for all $\eps\in(0,\eps_0]$ with operator norm uniformly bounded in $\eps$.

Let now $(F_\eps)_\eps$ be a net in $\mM_{L^2_{\rm{loc}}(\R^n)}$ representing $F\in\LL_{2,{\rm{loc}}}(\Gc(\R^n),\wt{\C})$. We have that $(\phi F_\eps)_\eps\in\mM_{L^2(\Om)}$ and we can define for $\eps\in(0,\eps_0]$ the net
\beq
\label{def_sol}
T_\eps:=\wt{q}_\eps(\wt{I}+\wt{r}_\eps)^{-1}(\phi F_\eps).
\eeq
$T_\eps$ belongs to $L^2(\Om)$ for all $\eps\in(0,\eps_0]$ and the properties of the operators involved in \eqref{def_sol} yield
\begin{multline*}
\Vert T_\eps\Vert_2=\Vert\wt{q}_\eps(\wt{I}+\wt{r}_\eps)^{-1}(\phi F_\eps)\Vert_2=\Vert q_\eps(x,D)(\phi(\wt{I}+\wt{r}_\eps)^{-1}(\phi F_\eps))|_\Om\Vert_2\\
\le c|q_\eps|^{(0)}_{l_0}\Vert\phi(\wt{I}+\wt{r}_\eps)^{-1}(\phi F_\eps)\Vert_2\le c'|q_\eps|^{(0)}_{l_0}\Vert(\wt{I}+\wt{r}_\eps)^{-1}\Vert\, \Vert\phi F_\eps\Vert_2.  
\end{multline*}
This means that, restricting $\eps$ on the interval $(0,\eps_0]$, the net $(T_\eps)_\eps$ is $L^2(\Om)$-moderate and therefore generates a basic functional $T$ in $\LL(\Gc(\Om),\wt{\C})$. $T$ solves the equation $P(x,D)T=\phi F$ on $\Om$. Indeed, working at the level of the representatives we have
\begin{multline*}
P_\eps(x,D)|_\Om(\wt{q}_\eps(\wt{I}+\wt{r}_\eps)^{-1}(\phi F_\eps))=P_\eps(x,D)|_\Om(q_\eps(x,D)\phi(\wt{I}+\wt{r}_\eps)^{-1}(\phi F_\eps))\\
=P_\eps(x,D)q_\eps(x,D)|_\Om(\phi(\wt{I}+\wt{r}_\eps)^{-1}(\phi F_\eps))=(\wt{I}+\wt{r}_\eps)(\wt{I}+\wt{r}_\eps)^{-1}(\phi F_\eps)=\phi F_\eps.
\end{multline*}
\end{proof}
\begin{remark}
\label{rem_set_up}
It is not restrictive to consider differential operators with coefficients in $\G_\infty(\R^n)$ when one wants to investigate local solvability in the Colombeau context. Indeed, if we assume to work on an open subset $\Om'$ and we take $P(x,D)=\sum_{|\alpha|\le m}c_\alpha(x)D^\alpha$ with $c_\alpha\in\G(\Om')$, by choosing the neighborhood $\Om$ of $x_0$ small enough the equation $P(x,D)T=\phi F$ on $\Om$ is equivalent to $P_1(x,D)T=\phi F$ with
\[
P_1(x,D)=\sum_{|\alpha|\le m}\varphi(x)c_\alpha(x)D^\alpha
\]
and $\varphi\in\Cinfc(\Om')$ identically $1$ on $\Om$. It follows that $\varphi c_\alpha\in\Gc(\Om')\subseteq\G_\infty(\R^n)$ and therefore we are in the mathematical set-up of Theorem \ref{theo_sol_par}.
\end{remark}
In the next proposition we find a family of differential operators which satisfy the hypotheses of Theorem \ref{theo_sol_par}, in other words a condition on the symbol which assures the existence of a parametrix $q$ with regularizing term $r$ as above. We go back to some definition of generalized hypoelliptic symbol introduced for pseudodifferential operators in \cite{Garetto:04}. Here the attention is focused not so much on the parametrix $q$ but on the required boundedness in $\eps$ of the regularizing operator $r$. This makes us to avoid some more general definitions of hypoelliptic symbol already employed in Colombeau theory, see \cite{Garetto:ISAAC07, GGO:03, GH:05}, which have less restrictive assumptions on the scales in $\eps$, guarantee the existence of a parametrix but not the desired behaviour of $r$.
\begin{proposition}
\label{prop_sol_hyp}
Let $P(x,D)=\sum_{|\alpha|\le m}c_\alpha(x)D^\alpha$ be a differential operator with coefficients $c_\alpha\in\G_\infty(\R^n)$. We assume that there exists $a,a'\in\R$, $a\le a'$, $0\le m'\le m$, $R>0$ and a representative $(P_\eps)_\eps$ of $P$ fulfilling the following conditions: 
\begin{itemize}
\item[(i)] $|P_\eps|^{(m)}_{\alpha,\beta}=O(\eps^{a})$ for all $\alpha,\beta\in\N^n$;
\item[(ii)] there exists $c>0$ such that 
$$|P_\eps(x,\xi)|\ge c\,\eps^{a'}\lara{\xi}^{m'}$$
for all $x\in\R^n$, for $|\xi|\ge R$ and for all $\eps\in(0,1]$;
\item[(iii)] for all $\alpha,\beta\in\N^n$ there exists $(c_{\alpha,\beta,\eps})_\eps$ with $c_{\alpha,\beta,\eps}=O(1)$ such that
\[
|\partial^\alpha_\xi\partial^\beta_x P_\eps(x,\xi)|\le c_{\alpha,\beta,\eps}|P_\eps(x,\xi)|\lara{\xi}^{-|\alpha|}
\]
for all $x\in\R^n$, for $|\xi|\ge R$ and for all $\eps\in(0,1]$.
\end{itemize}
Then there exists $(q_\eps)_\eps\in\mM_{S^{-{m'}}(\R^{2n}),-a'}$ and $(r_\eps)_\eps\in\mM_{S^{-\infty}(\R^{2n}),a-a'}$ such that 
\[
P_\eps(x,D)q_\eps(x,D)=I+r_\eps(x,D)
\]
for all $\eps\in(0,1]$. Moreover there exists $s_\eps(x,D)$ with $(s_\eps)_\eps\in\mM_{S^{-\infty}(\R^{2n}),2a-2a'}$ such that
\[
q_\eps(x,D)P_\eps(x,D)=I+s_\eps(x,D)
\]
for all $\eps\in(0,1]$.
\end{proposition}
\begin{proof}
Let $\psi$ be a smooth function in the variable $\xi$ such that $\psi(\xi)=0$ for $|\xi|\le R$ and $\psi(\xi)=1$ for $|\xi|\ge 2$. By adapting the proof of Proposition 8.1 and Theorem 8.1 in \cite{Garetto:04} to our situation one easily obtains from the hypotheses $(i)$ and $(ii)$ that 
\begin{itemize}
\item[-] $q_{0,\eps}:=\psi(\xi)P_\eps^{-1}(x,\xi)$ defines a net in $\mM_{S^{-m'}(\R^{2n}),-a'}$,
\item[-] $(q_{0,\eps}\partial^\alpha_\xi\partial^\beta_x P_\eps)_\eps\in \mM_{S^{-|\alpha|}(\R^{2n}),0}$ for all $\alpha,\beta\in\N^n$,
\item[-] for each $j\ge 1$, the net
\[
q_{j,\eps}:=-\biggl\{\sum_{|\gamma|+l=j,\, l<j}\frac{(-i)^{|\gamma|}}{\gamma !}\partial^\gamma_\xi P_\eps\partial^\gamma_x q_{l,\eps}\biggr\}q_{0,\eps}
\]
belongs to $\mM_{S^{-m'-j}(\R^{2n}),-a'}$.
\end{itemize}
Corollary \ref{corol_asymp} implies that there exists $(q_\eps)_\eps\in\mM_{S^{-m'}(\R^{2n}),-a'}$ having $\{(q_{j,\eps})_\eps\}_j$ as asymptotic expansion with fixed moderateness $\eps^{-a'}$. Let us now consider the composition $P_\eps(x,D)q_\eps(x,D)=\lambda_\eps(x,D)$. Basic properties of symbolic calculus show that 
\[
\biggl(\lambda_\eps-\sum_{|\gamma|<r}\frac{(-i)^{|\gamma|}}{\gamma !}\partial^\gamma_\xi P_\eps\partial^\gamma_x q_\eps\biggr)_\eps\in\mM_{S^{m-m'-r}(\R^{2n}),a-a'}
\]
for all $r\ge 1$. Making use of $(q_\eps-\sum_{l=0}^{r-1}q_{l,\eps})_\eps\in\mM_{S^{-m'-r}(\R^{2n}),-a'}$ we can write
\begin{multline*}
\sum_{|\gamma|<r}\frac{(-i)^{|\gamma|}}{\gamma !}\partial^\gamma_\xi P_\eps\partial^\gamma_x q_\eps = \sum_{|\gamma|<r}\sum_{l=0}^{r-1}\frac{(-i)^{|\gamma|}}{\gamma !}\partial^\gamma_\xi P_\eps\partial^\gamma_x q_{l,\eps} +s_\eps = P_\eps q_{0,\eps}+\sum_{j=1}^{r-1}P_\eps q_{j,\eps}\\
+\sum_{j=1}^{r-1}\sum_{|\gamma|+l=j,\, l<j}\frac{(-i)^{|\gamma|}}{\gamma !}\partial^\gamma_\xi P_\eps\partial^\gamma_x q_{l,\eps}+\sum_{|\gamma|+l\ge r,\, |\gamma|<r,\, l<r}\frac{(-i)^{|\gamma|}}{\gamma !}\partial^\gamma_\xi P_\eps\partial^\gamma_x q_{l,\eps}+s_\eps,
\end{multline*}
where $(s_\eps)_\eps\in\mM_{S^{m-m'-r}(\R^{2n}),a-a'}$. By definition of $q_{0,\eps}$ and $q_{j,\eps}$ we have that the right-hand side of the previous formula equals 
\[
1+\sum_{|\gamma|+l\ge r,\, |\gamma|<r,\, l<r}\frac{(-i)^{|\gamma|}}{\gamma !}\partial^\gamma_\xi P_\eps\partial^\gamma_x q_{l,\eps}+s_\eps
\] 
when $|\xi|\ge 2R$. Hence the net $(\lambda_\eps-1)_\eps$ belongs to $\mM_{S^{m-m'-r}(\R^{2n}),a-a'}$ for $|\xi|\ge 2R$. Since $q_{0,\eps}P_\eps(x,\xi)-1=\psi(\xi)-1\in\Cinfc(\R^n)$ the domain restriction can be dropped. Concluding, $(\lambda_\eps-1)_\eps:=(r_\eps)_\eps$ is an element of $\mM_{S^{-\infty}(\R^{2n}),a-a'}$. Analogously one can construct a net of symbols $(q'_\eps)_\eps\in\mM_{S^{-m'}(\R^{2n}),-a'}$ such that $q'_\eps(x,D)P_\eps(x,D)=I+r'_\eps(x,D)$ with $(r'_\eps)_\eps\in\mM_{S^{-\infty}(\R^{2n}),a-a'}$. By applying $q'_\eps(x,D)P_\eps(x,D)$ to $q_\eps(x,D)$ we get
\[
(I+r'_\eps(x,D))q_\eps(x,D)=q'_\eps(x,D)P_\eps(x,D)q_\eps(x,D)=q'_\eps(x,D)(I+r_\eps(x,D))
\]
which at the level of symbols means
\[
q_\eps+r'_\eps\sharp q_\eps=q'_\eps+q'_\eps\sharp r_\eps.
\]
Thus, $(q_\eps-q'_\eps)\in\mM_{S^{-\infty}(\R^{2n}),a-2a'}$ and since $a-a'\ge 2a-2a'$ the equality $q_\eps(x,D)P_\eps(x,D)=I+s_\eps(x,D)$ holds with $(s_\eps)_\eps\in\mM_{S^{-\infty}(\R^{2n}),2a-2a'}$. 
\end{proof}
A straightforward combination of Proposition \ref{prop_sol_hyp} with Theorem \ref{theo_sol_par} entails the following result of local solvability.
\begin{proposition}
\label{prop_sol_hyp_2}
Let $P(x,D)=\sum_{|\alpha|\le m}c_\alpha(x)D^\alpha$ be a differential operator with coefficients $c_\alpha\in\G_\infty(\R^n)$. Let $(P_\eps)_\eps$ be a representative of $P$ fulfilling the hypotheses of Proposition \ref{prop_sol_hyp} with $a=a'$. Then for all $x_0\in\R^n$ there exists a neighborhood $\Om$ of $x_0$ and a cut-off function $\phi$, identically $1$ near $x_0$, such that  
\[
\forall F\in\LL_{2,{\rm{loc}}}(\Gc(\R^n),\wt{\C})\ \exists T\in\LL(\Gc(\Om),\wt{\C})\qquad\qquad P(x,D)T=\phi F\qquad \text{on $\Om$}.
\]
\end{proposition}
\begin{proof}
If $a=a'$ from Proposition \ref{prop_sol_hyp} we have that there exists a parametrix $q_\eps(x,D)$ with $(q_\eps)_\eps\in\mM_{S^{-m'}(\R^{2n}),-a}$, $-m'\le 0$, and a regularizing operator $r_\eps(x,D)$ with $(r_\eps)_\eps\in\mM_{S^{-\infty}(\R^{2n}),0}$. This means that $|r_\eps|^{(l)}_{0,0}=O(1)$ for all $l\in\R$. The conditions under which Theorem \ref{theo_sol_par} holds are therefore fulfilled.
\end{proof}
\begin{example}
As an explanatory example we consider the operator generated by
\[
P_\eps(x,D)=-\eps^a\Delta+\sum_{|\alpha|\le 1}c_{\alpha,\eps}(x)D^\alpha,
\]
where $\Delta=\sum_{i=1}^n\frac{\partial^2}{\partial_{x_i^2}}$,
\[
c_{\alpha,\eps}=c_\alpha\ast\varphi_{\omega(\eps)},\qquad c_\alpha\in L^\infty(\R^n),
\]
$\varphi$ is a mollifier in $\S(\R^n)$ and $(\omega^{-1}(\eps))_\eps$ a slow scale net. It follows that $[(c_{\alpha,\eps})_\eps]$ belongs to $\G_\infty(\R^n)$ with
\[
\Vert\partial^\beta c_{\alpha,\eps}\Vert_\infty\le \Vert c_\alpha\Vert_\infty\, \omega(\eps)^{-|\beta|}\Vert\partial^\beta\varphi\Vert_1\le c\eps^{-b}
\]
for all $b>0$. For any $a<0$ this operator is locally solvable in the sense of Theorem \ref{theo_sol_par} because it fulfills the conditions $(i)$, $(ii)$, $(iii)$ of Proposition \ref{prop_sol_hyp} with $a=a'$.
\end{example}
Due to the existence of a generalized parametrix for the operator $P(x,D)$ of Proposition \ref{prop_sol_hyp_2}, the local solution inherits the regularity properties of the right hand-side. 
\begin{proposition}
\label{prop_sol_hyp_3}
Let $P(x,D)=\sum_{|\alpha|\le m}c_\alpha(x)D^\alpha$ be a differential operator with coefficients $c_\alpha\in\G_\infty(\R^n)$. Let $(P_\eps)_\eps$ be a representative of $P$ fulfilling the hypotheses of Proposition \ref{prop_sol_hyp} with $a=a'$. Then for all $x_0\in\R^n$ there exists a neighborhood $\Om$ of $x_0$ and a cut-off function $\phi$, identically $1$ near $x_0$, such that  
\[
\forall f\in\G(\R^n)\ \exists u\in\G(\Om)\qquad\qquad P(x,D)u=\phi f\qquad \text{on $\Om$},
\]
and
\[
\forall f\in\Ginf(\R^n)\ \exists u\in\Ginf(\Om)\qquad\qquad P(x,D)u=\phi f\qquad \text{on $\Om$}.
\]
\end{proposition} 
\begin{proof}
First $f\in\G(\R^n)$ can be regarded as an element of $\LL_{2,{\rm{loc}}}(\Gc(\R^n),\wt{\C})$. Let $(f_\eps)_\eps$ be a representative of $f$. From Theorem \ref{theo_sol_par} we can construct a local solution $u\in\LL(\Gc(\Om),\wt{\C})$ having a representative  
\[
u_\eps=\wt{q_\eps}(\wt{I}+\wt{r_\eps})^{-1}(\phi f_\eps)
\]
in $\mM_{L^2(\Om)}$. The equality 
\[
q_\eps(x,D)P_\eps(x,D)=I+s_\eps(x,D)
\]
holds on $\S'(\R^n)$. Taking the restrictions of the previous operators to the open set $\Om$, since they all map $L^2(\Om)$ into $L^2(\Om)$ we have that
\beq
\label{formula}
q_\eps(x,D)P_\eps(x,D)v = v+s_\eps(x,D)v
\eeq
holds on $\Om$ for all $v\in L^2(\Om)$. Here and in the sequel we omit the restriction notation $|_\Om$ for the sake of simplicity. From \eqref{formula} follows 
\[
u_\eps+s_\eps(x,D)u_\eps=q_\eps(x,D)(\phi f_\eps).
\]
Since $s_\eps(x,D)$ is a regularizing operator and $(f_\eps)_\eps$ is a net of smooth functions we already see that $(u_\eps)_\eps$ is a net of smooth functions as well. In particular, from the mapping properties of generalized pseudodifferential operators we know that $(q_\eps(x,D)(\phi f_\eps))_\eps\in\mM_{\Cinf(\Om)}=\EM(\Om)$. Finally we write $s_\eps(x,D)u_\eps$ as
\[
\int_\Om k_{s_\eps}(x,y)u_\eps(y)\, dy.
\]
Combining the boundedness of the open set $\Om$ with the following kernel property 
\[
\forall\alpha\in\N^n\, \forall d\in \N\qquad  \sup_{(x,y)\in\R^{2n}}\lara{x}^{-d}\lara{y}^d|\partial^{\alpha}_x\partial^{\beta}_y k_{s_\eps}(x,y)|=O(1) 
\]
we obtain
\beq
\label{regularizing}
\sup_{x\in\Om}|\partial^\alpha s_\eps(x,D)u_\eps|\le \sup_{x\in\Om}\Vert \partial^\alpha_x k_{s_\eps}(x,\cdot)\Vert_{L^2(\Om)}\, \Vert u_\eps\Vert_{L^2(\Om)}\le c\Vert u_\eps\Vert_{L^2(\Om)}
\eeq
for $\eps$ small enough. Hence $(s_\eps(x,D)u_\eps)_\eps\in\EM(\Om)$. Concluding the net $(u_\eps)_\eps$ belongs to $\EM(\Om)$ and generates a solution $u$ in $\G(\Om)$ to $P(x,D)u=\phi f$.

When $f\in\Ginf(\R^n)$ since the net of symbols $(q_\eps)_\eps$ is regular we have that $(q_\eps(x,D)(\phi f_\eps))_\eps$ generates an element of $\Ginf(\Om)$. Clearly as one sees in \eqref{regularizing} also $s(x,D)u$ belongs to $\Ginf(\Om)$. Hence $u\in\Ginf(\Om)$.
\end{proof}

\section{Local solvability of partial differential operators $\G$-elliptic in a neighborhood of a point}
\label{sec_elliptic}
In this section we concentrate on a special type of partial differential operators with coefficients in $\G(\R^n)$. Their properties will inspire the more general model introduced in Section 4. In the sequel we often refer to the work on generalized hypoelliptic and elliptic symbols in \cite{Garetto:ISAAC07, Garetto:08c, GGO:03}.
\begin{definition}
\label{def_ellip_point}
Let $P(x,D)=\sum_{|\alpha|\le m}c_\alpha(x)D^\alpha$ be a partial differential operator with coefficients in $\G(\R^n)$. We say that $P(x,D)$ is $\G$-elliptic in a neighborhood of $x_0$ if there exists a representative $(P_{m,\eps})_\eps$ of the principal symbol $P_{m}$, a neighborhood $\Om$ of $x_0$, $a\in\R$ and $c>0$ such that
\beq
\label{def_ellip}
|P_{m,\eps}(x,\xi)|\ge c\eps^a
\eeq
for all $x\in\Om$, for $|\xi|=1$ and for all $\eps\in(0,1]$.
\end{definition}
\begin{remark}
\label{rem_ellip}
It is clear that \eqref{def_ellip} holds for an arbitrary representative $(P_\eps)_\eps$ of $P$ on a smaller interval $(0,\eta]$ and with some smaller constant $c>0$. In addition if $P(x,D)$ is $\G$-elliptic in a neighborhood of $x_0$ then $P(x_0,D)$ is $\G$-elliptic. The converse does not hold. Indeed, let $\varphi\in\Cinfc(\R)$ with $\varphi(0)=1$ and $\varphi(x)=0$ for $|x|\ge 2$. The differential operator $P(x,D)$ with representative $P_\eps(x,D)=\varphi(x/\eps)D^2$ is $\G$-elliptic in $0$ but not in a neighborhood $\{|x|<r\}$ of $0$. This is due to the fact that $P_\eps(x,\xi)=0$ for $x\neq 0$ and $\eps<2^{-1}|x|$.
\end{remark}
As for $\G$-elliptic operators with constant Colombeau coefficients (see \cite[Section6]{Garetto:08c}) the following estimates hold in a neighborhood of $x_0$.
\begin{proposition}
\label{prop_ellip_point}
Let $P(x,D)$ be $\G$-elliptic in a neighborhood of $x_0$. Then there exists a representative $(P_\eps)_\eps$ of $P$, a neighborhood $\Om$ of $x_0$, moderate strictly nonzero nets $(R_\eps)_\eps$ and $(c_{\alpha,\beta,\eps})_\eps$ and a constant $c_0>0$ such that
\[
|P_{\eps}(x,\xi)|\ge c_0\eps^a\lara{\xi}^m
\]
and
\[
|\partial^\alpha_\xi\partial^\beta_x P_{\eps}(x,\xi)|\le c_{\alpha,\beta,\eps}|P_\eps(x,\xi)|\lara{\xi}^{-|\alpha|}
\]
for $x\in\Om$, $|\xi|\ge R_\eps$ and for all $\eps\in(0,1]$.
\end{proposition}
\begin{proof}
From Definition \ref{def_ellip_point} we have
\[
|P_\eps(x,\xi)|\ge |P_{m,\eps}(x,\xi)|-|P_{\eps}(x,\xi)-P_{m,\eps}(x,\xi)|\ge c\eps^a|\xi|^m-c_{m-1,\eps}\lara{\xi}^{m-1}
\]
for all $\xi\in\R^n$, $x\in\Om$, $\eps\in(0,1]$ and with $(c_{m-1,\eps})_\eps$ a moderate and strictly nonzero net. Defining the radius $R_\eps=\max\{1,2^m c_{m-1,\eps}c^{-1}\eps^{-a}\}$ we get for $x\in\Om$, $|\xi|\ge R_\eps$ and for all $\eps$, the inequality
\[
|P_\eps(x,\xi)|\ge |\xi|^m(c\eps^a-c_{m-1,\eps}2^{m-1}|\xi|^{-1})\ge \frac{c}{2}\eps^a|\xi|^m\ge c_0\eps^a\lara{\xi}^m.
\]
Concerning the derivatives we have, always for $|\xi|\ge R_\eps$,
\[
|\partial^\alpha_\xi\partial^\beta_x P_{\eps}(x,\xi)|\le \lambda_{\alpha,\beta,\eps}\lara{\xi}^{m-|\alpha|}\le \lambda_{\alpha,\beta,\eps}c_0^{-1}\eps^{-a}|P_\eps(x,\xi)|\lara{\xi}^{-|\alpha|}=c_{\alpha,\beta,\eps}|P_\eps(x,\xi)|\lara{\xi}^{-|\alpha|}.
\]
\end{proof}
By adapting the arguments of Proposition \ref{prop_sol_hyp} to this kind of nets of symbols, and in analogy with \cite[Theorem 6.8]{GGO:03}, \cite[Propositions 2.7, 2.8]{Garetto:ISAAC07}, we obtain that a differential operator $\G$-elliptic in a neighborhood of $x_0$ admits a local parametrix.
\begin{proposition}
\label{prop_par_point}
Let $P(x,D)$ be $\G$-elliptic in a neighborhood of $x_0$. Then there exists a neighborhood $\Om$ of $x_0$ and generalized symbols $q\in\G_{S^{-m}(\Om\times\R^n)}$ and $r,s\in\G_{S^{-\infty}(\Om\times\R^n)}$ such that
\beq
\label{formula_1}
\begin{split}
P(x,D)q(x,D)&=I+r(x,D),\\
q(x,D)P(x,D)&=I+s(x,D)
\end{split}
\eeq
as operators acting on $\Gc(\Om)$ with values in $\G(\Om)$.
\end{proposition}
Note that if we take $\Om$ bounded we can assume that the estimates involving the symbols in \eqref{formula_1} are global in $\xi$ and $x$ as well. The equalities between generalized operators followed from the corresponding equalities at the level of representatives. More precisely,
\[
\begin{split}
P_\eps(x,D)q_\eps(x,D)&=I+r_\eps(x,D),\\
q_\eps(x,D)P_\eps(x,D)&=I+s_\eps(x,D)
\end{split}
\]
for all $\eps\in(0,1]$. Since $P(x,D)$ is properly supported $q(x,D)P(x,D)=I+s(x,D)$ holds on $\G(\Om)$ and $\LL(\Gc(\Om),\wt{\C})$ as well. It follows that if $P(x,D)$ is $\G$-elliptic in a neighborhood of $x_0$ and locally solvable then it inherits the regularity of the right-hand side, in the sense that if $P(x,D)T=v$ on $\Om$ with $T\in\LL(\Gc(\Om),\wt{\C}$ and $v\in\G(\Om)$ then $T\in\G(\Om)$. The problem is that Definition \ref{def_ellip_point} in general does not guarantee the assumption on the parametrix $(q_\eps)_\eps$ and the regularizing term $(r_\eps)_\eps$ which allow to apply Theorem \ref{theo_sol_par} and obtain local solvability. In the following particular case an operator which is $\G$-elliptic in a neighborhood of $x_0$ is also locally solvable at $x_0$.
\begin{proposition}
\label{prop_loc_sol_ellip_point}
Let $P(x,D)=\sum_{|\alpha|\le m}c_\alpha(x)D^\alpha$ be $\G$-elliptic in a neighborhood of $x_0$ with
\[
|P_{m,\eps}(x,\xi)|\ge c\eps^a
\]
in a neighborhood $\Om_1$ of $x_0$, for all $\xi\in\R^n$ with $|\xi|=1$ and for all $\eps\in(0,1]$. If the coefficients $c_\alpha$ are $\Ginf$-regular in $x_0$ of order $a$, i.e. on a neighborhood $\Om_2$ of $x_0$ the following
\[
\forall\beta\in\N^n\,\qquad \sup_{x\in\Om_2}|\partial^\beta c_{\alpha,\eps}(x)|=O(\eps^a)
\]
holds, then there exist a neighborhood $\Om$ of $x_0$ and a cut-off function $\phi\in\Cinfc(\Om)$ identically $1$ near to $x_0$ such that:
\begin{itemize}
\item[(i)] for all $F\in\LL_{2,{\rm{loc}}}(\Gc(\R^n),\wt{\C})$ there exist $T\in\LL(\Gc(\Om),\wt{\C})$ solving $P(x,D)T=\phi F$ on $\Om$;
\item[(ii)] for all $f\in\G(\R^n)$ there exists $u\in\G(\Om)$ solving $P(x,D)u=\phi f$ on $\Om$;
\item[(iii)] for all $f\in\Ginf(\R^n)$ there exists $u\in\Ginf(\Om)$ solving $P(x,D)u=\phi f$ on $\Om$;
\end{itemize}
\end{proposition}
\begin{proof}
We can choose a representative $(P_\eps)_\eps$ such that the inequalities $|P_{m,\eps}(x,\xi)|\le c\eps^a$ and $|P_{\eps}(x,\xi)-P_{m,\eps}(x,\xi)|\le c_{m-1}\eps^a\lara{\xi}^{m-1}$ hold on the interval $(0,1]$, for all $x$ in a  neighborhood of $x_0$ and all $\xi\in\R^n$. By following the proof of Proposition \ref{prop_ellip_point} we see that the radius does not depend on $\eps$ and that the nets $(c_{\alpha,\beta,\eps})$ are $O(1)$ as $\eps$ tends to $0$. We are under the hypotheses of Proposition \ref{prop_sol_hyp_2}. This yields the first assertion. Proposition \ref{prop_par_point} and the considerations above on the regularity of $P(x,D)$ prove assertion $(ii)$. Finally, assertion $(iii)$ is clear for Proposition \ref{prop_sol_hyp_3}.
\end{proof}
We have found a class of partial differential operators with coefficients in $\G(\R^n)$, that under the hypothesis of $\G$-ellipticity in a neighborhood $x_0$ and under suitable assumptions on the moderateness of the coefficients, are locally solvable at $x_0$. These locally solvable operators belong to the wider family of operators which can be written in the form
\beq
\label{const_str_form}
P_0(D)+\sum_{j=1}^r c_j(x)P_j(D),
\eeq 
in a neighborhood of $x_0$. Here the operators $P_0(D)$, $P_j(D)$, $j=1,...,r$ have constant Colombeau coefficients and each $c_j$ is a Colombeau generalized function. 

We conclude this section by proving that a differential operator $P(x,D)$ with coefficients in $\G$ which is $\G$-elliptic in $x_0$, i.e. $P(x_0,D)$ is $\G$-elliptic, can be written in the form \eqref{const_str_form}.
\begin{proposition}
\label{prop_ellip_cs}
Let $P(x,D)=\sum_{|\alpha|\le m}c_\alpha(x)D^\alpha$ be a differential operator with coefficients in $\G(\R^n)$ which is $\G$-elliptic in $x_0$. Then $P(x,D)$ can be written in the form \eqref{const_str_form} with
\begin{itemize}
\item[-] $c_j\in\G(\R^n)$, $c_j(x_0)=0$, $P_0(D)$ and $P_j(D)$ operators with constant Colombeau coefficients,
\item[-] $\wt{P_0}$ invertible in some point of $\R^n$,
\item[-] $P_0(D)$ stronger than any $P_j(D)$.
\end{itemize}
\end{proposition}
\begin{proof}
We set $P_0(D)=P(x_0,D)$ and we have
\[
P(x,D)=P_0(D)+\sum_{|\alpha|\le m}(c_\alpha(x)-c_\alpha(x_0))D^\alpha.
\]
Clearly the coefficients $c_\alpha(x)-c_\alpha(x_0)$ belong to $\G(\R^n)$ and vanish for $x=x_0$. For each $\alpha\in\N^n$ we find an operator $P_{j(\alpha)}(D)=D^\alpha$. By hypothesis $P_0(D)$ is $\G$-elliptic in $x_0$. Hence from Proposition \ref{prop_g_ellip} we have that $P_0(D)$ is stronger than any $P_{j(\alpha)}(D)$ with $|\alpha|\le m$ and the weight function $\wt{P_{0}}$ is invertible in any point of $\R^n$.
\end{proof}
A differential operator which is $\G$-elliptic in a neighborhood $\Om$ of $x_0$ is in particular $\G$-elliptic in $x_0$ and therefore it can be written in the form \eqref{const_str_form} on the whole of $\R^n$. The special structure \eqref{const_str_form} of the $\G$-elliptic operators motivates the investigations of Section \ref{sec_bounded}.

\section{Bounded perturbations of differential operators with constant Colombeau coefficients: definition and examples}
\label{sec_bounded}
In this section we concentrate on operators with coefficients in $\G(\R^n)$ which are locally a bounded perturbation of a differential operator with constant Colombeau coefficients as in \eqref{const_str_form}. More precisely, we say that $P(x,D)=\sum_{|\alpha|\le m}c_\alpha(x)D^\alpha$ is of \emph{bounded perturbation type}, or of \emph{BP-type}, in a neighborhood $\Om$ of $x_0$ if it has the form 
\[
P_0(D)+\sum_{j=1}^r c_j(x)P_j(D),
\]
when restricted to $\Om$, with
\begin{itemize}
\item[(h1)]$c_j\in\G(\Om)$, $c_j(x_0)=0$, $P_0(D)=P(x_0,D)$ and $P_j(D)$, $j=1,...,r$, operators with constant Colombeau coefficients
\item[(h2)]$\wt{P_0}$ invertible in some point of $\R^n$,
\item[(h3)]$P_0(D)$ stronger than any $P_j(D)$.
\end{itemize}
\begin{remark}
Our definition of BP-type is clearly inspired by the classical theory of operators of constant strength (see \cite{Hoermander:63, Hoermander:V2}). The direct generalization of this concept to the Colombeau setting would mean to require $P(x_0,D)\prec P(x,D)\prec P(x_0,D)$ for all $x$ in neighborhood $\Om$ of $x_0$ with $\prec$ the order relation introduced in Section 1. However, due to some some structural and technical constraints of our framework, it is not clear at the moment if one can obtain from this general definition a local bounded perturbation property as above. This is related to the fact that one can not use the properties of a linear space on the set of differential operators with coefficients in $\wt{\C}$ weaker than $P_0(D)$. Indeed this set has the algebraic structure of a module over $\wt{\C}$ and $\wt{\C}$ is only a ring and not a field. 
\end{remark}
As for the operators with constant Colombeau coefficients (Theorem 7.8 in \cite{HO:03}) the local solvability of $P(x,D)$ in the Colombeau algebra $\G(\Om)$, where $\Om$ is an open neighborhood of $x_0$, implies the invertibility of the weight function $\wt{P_0}$ in some point of $\R^n$.
\begin{proposition}
\label{prop_nec_cond}
Let $P(x,D)$ be a differential operator with coefficients in $\G(\R^n)$ such that has the form
\[
P_0(D)+\sum_{j=1}^r c_j(x)P_j(D),
\]
in a neighborhood $\Om$ of $x_0$ and fulfills the hypothesis $(h1)$. Let $v\in\G(\Om)$ with $v(x_0)$ invertible in $\wt{\C}$. If the equation $P(x,D)u=v$ is solvable in $\G(\Om)$ then $\wt{P_0}$ is invertible in some point of $\R^n$.
\end{proposition}
\begin{proof}
We begin by observing that
\[
v(x_0)=P(x_0,D)u(x_0)=P_0(D)u(x_0).
\]
From \eqref{est_inv} we see that $\wt{P_0}$ is invertible in some point of $\R^n$ if and only if it is invertible in any point of $\R^n$ and then in particular in $\xi=0$. We assume that $\wt{P_0}(0)$ is not invertible. It follows that for all $q$ there exists $\eps_q\in(0,q^{-1}]$ such that 
\[
\wt{P_{\eps_q}}^2(0)=\sum_{|\alpha|\le m}|c_{\alpha,\eps_q}(x_0)|^2(\alpha !)^2<\eps_q^q.
\]
Choosing $\eps_q\searrow 0$ we have
\[
|c_{\alpha,\eps_\nu}(x_0)|^2\le (\alpha!)^{-2}\eps_\nu^\nu\le (\alpha!)^{-2}\eps_\nu^q
\]
for all $\nu\ge q$. Hence, for $c_\eps=1$ for $\eps=\eps_q$, $q\in\N$, and $c_\eps=0$ otherwise, all the nets $(c_\eps\cdot c_{\alpha,\eps}(x_0))_\eps$ are negligible. Concluding, for $c=[(c_\eps)_\eps]\in\wt{\R}$ the equality $v(x_0)=P_0(D)u(x_0)$ implies
\[
c\cdot v(x_0)=c\cdot P_0(D)u(x_0)=\sum_{|\alpha|\le m}(c\cdot c_\alpha(x_0))D^\alpha u(x_0)=0,
\]
in contradiction with $c\cdot v(x_0)\neq 0$.
\end{proof}

We now collect some examples of operators of BP-type. It is clear by Proposition \ref{prop_ellip_cs} that the differential operators which are $\G$-elliptic in $x_0$ are of BP-type in any neighborhood of $x_0$. For the advantage of the reader we write two explicit examples.
\begin{example}
\leavevmode
\begin{trivlist}
\item[(i)] Let $c_i\in\G(\R)$, $i=0,...,3$, with $c_2(0)$ invertible in $\wt{\C}$, $\supp\, c_3\subseteq(-3/2,-1/2)$ and $\supp\, c_2\subseteq(-1,1)$. The operator 
\[
P(x,D)=c_3(x)D^3+c_2(x)D^2+c_1(x)D+c_0(x)
\]
is a bounded perturbation of $P(0,D)$ in the neighborhood $\Om:=(-1/4,1/4)$. Indeed, $P(x,D)|_\Om = c_2|_\Om D^2+c_1|_\Om D+c_0|_\Om$ and $P(x,D)|_\Om$ is $\G$-elliptic in $0$.
\item[(ii)] For $i=1,2$ let $\varphi_{i}\in\Cinfc(\R)$, $\varphi_i(0)=1$ and $\varphi_{i,\eps}(x)=\varphi_i(x/\eps)$. Let $c_\alpha\in\G(\R)$ for $|\alpha|\le 1$. The operator $P(x,D)$ with representative
\[
P_\eps(x,D)=\varphi_{1,\eps}(x)D^2_{x_1}+\varphi_{2,\eps}(x)D^2_{x_2}+\sum_{|\alpha|\le 1}c_{\alpha,\eps}(x)D^\alpha
\] 
is a bounded perturbation of $P(0,D)$. More precisely we can write
\[
P_\eps(x,D)=D^2_{x_1}+D^2_{x_2}+(\varphi_{1,\eps}(x)-1)D_{x_1}^2+(\varphi_{2,\eps}(x)-1)D_{x_2}^2+\sum_{|\alpha|\le 1}(c_{\alpha,\eps}(x)-c_{\alpha,\eps}(0))D^\alpha.
\]
\end{trivlist}
\end{example} 
A statement analogous to Proposition \ref{prop_ellip_cs} holds for operators of principal type.
\begin{proposition}
\label{prop_princ_cs}
Let $P(x,D)=\sum_{|\alpha|\le m}c_\alpha(x)D^\alpha$ be a differential operator with coefficients in $\G(\R^n)$. Let the coefficients of the principal part be constant with at least one of them invertible in $\wt{\C}$. If $P(x,D)$ is of principal type then it is of BP-type in any neighborhood $\Om$ of any point $x_0\in\R^n$.
\end{proposition}
\begin{proof}
We take $P_0(D)=P(x_0,D)=P_m(D)+\sum_{|\alpha|\le m-1}c_\alpha(x_0)D^\alpha$. By hypothesis we have that $P_0(D)$ is of principal type. Moreover, the invertibility of one of the coefficients of $P_m(D)$ entails the invertibility of $\wt{P_0}(\xi)$. Now we write
\[
P(x,D)=P_0(D)+\sum_{|\alpha|\le m-1}(c_\alpha(x)-c_\alpha(x_0))D^\alpha.
\]
The operators $P_{j(\alpha)}(D)=D^\alpha$ are all of order $\le m-1$ and therefore by Proposition \ref{prop_princ_type}$(i)$ $P_0(D)$ dominates any $P_{j(\alpha)}$.
\end{proof}
In the next proposition we see an example of operator of BP-type where the decomposition is obtained by deriving with respect to $\xi$.
\begin{proposition}
\label{prop_deriv}
Let $P(x,\xi)=c_{2,0}\xi_1^2+c_{1,1}\xi_1\xi_2+c_{0,2}\xi_2^2+c_{1,0}(x)\xi_1+c_{0,1}(x)\xi_2+c_{0,0}(x)$ be a polynomial in the $\R^2$-variable $(\xi_1,\xi_2)$ with coefficients in $\G(\R^2)$. Let the coefficients of the principal part be constant with at least one invertible. Let $x_0\in\R^2$ and $P_0(D)=P(x_0,D)$. If $4c_{2,0}c_{0,2}-c_{1,1}^2$ and $2c_{2,0}+2c_{0,2}+c_{1,1}$ are invertible in $\wt{\C}$ then for $j=1,2,3$ there exists $c_j\in\G(\R^2)$ with $c_j(x_0)=0$ such that
\[
P(x,D)=P_0(D)+c_1(x)P_1(D)+c_2(x)P_2(D)+c_3(x)P_3(D),
\]
where
\[
\begin{split}
P_1(D)&=P_0^{(1,0)}(D),\\
P_2(D)&= P_0^{(0,1)}(D),\\
P_3(D)&= P_0^{(2,0)}(D)+P^{(1,1)}_0(D)+P^{(0,2)}_0(D).
\end{split}
\]
Hence, $P(x,D)$ is of BP-type in any neighborhood of $x_0$.
\end{proposition}
\begin{proof}
We argue at the level of symbols. By fixing $x=x_0$ we have
\[
\begin{split}
P_{0}(\xi)&=P(x_0,\xi)=c_{2,0}\xi_1^2+c_{1,1}\xi_1\xi_2+c_{0,2}\xi_2^2+c_{1,0}(x_0)\xi_1+c_{0,1}(x_0)\xi_2+c_{0,0}(x_0)\\
P_{0}^{(1,0)}(\xi)&= 2c_{2,0}\xi_1+c_{1,1}\xi_2+c_{1,0}(x_0),\\
P_{0}^{(0,1)}(\xi)&= c_{1,1}\xi_1+2c_{0,2}\xi_2+c_{0,1}(x_0),\\
P_{0}^{(1,1)}(\xi)&=c_{1,1},\\
P_{0}^{(2,0)}(\xi)&= 2c_{2,0},\\
P_{0}^{(0,2)}(\xi)&=2c_{0,2}.
\end{split}
\]
From the invertibility of one of the principal part's coefficients we have that $\wt{P_0}$ is invertible in any point $\xi_0$ of $\R^n$. Moreover, $P_1:=P_{0}^{(0,1)}\prec\prec P_0$, $P_2:=P_{0}^{(0,1)}\prec\prec P_0$ and $P_3:=P_0^{(2,0)}+P^{(1,1)}_0+P^{(0,2)}_0\prec\prec P_0$. The operator 
\[
P(x,D)=P_0(D)+(c_{1,0}(x)-c_{1,0}(x_0))D_{x_1}+(c_{0,1}(x)-c_{0,1}(x_0))D_{x_2}+(c_{0,0}(x)-c_{0,0}(x_0))
\]
can be written as 
\[
P_0(D)+c_1(x)P_1(D)+c_2(x)P_2(D)+c_3(x)P_3(D),
\]
where the generalized functions $c_1(x)$, $c_2(x)$ and $c_3(x)$ are solutions of the following system:
\[
\begin{tabular}{ccccccc}
$2c_{2,0}\,c_1(x)$ & $+$ &$c_{1,1}\,c_2(x)$& {\,}& {\,} &$=$ &$c_{1,0}(x)-c_{1,0}(x_0)$,\\
$c_{1,1}\,c_1(x)$ & $+$ & $2c_{0,2}\,c_2(x)$& {\,}& {\,}& $=$ & $c_{0,1}(x)-c_{0,1}(x_0)$,\\
$c_{1,0}(x_0)\,c_1(x)$& $+$& $c_{0,1}(x_0)\,c_2(x)$& $+$& $(2c_{2,0}+2c_{0,2}+c_{1,1})c_3(x)$&$=$ &$c_{0,0}(x)-c_{0,0}(x_0)$.
\end{tabular}
\]
In detail, $c_1(x)=\frac{2c_{0,2}(c_{1,0}(x)-c_{1,0}(x_0))-c_{1,1}(c_{0,1}(x)-c_{0,1}(x_0))}{4c_{2,0}c_{0,2}-c_{1,1}^2}$, $c_2(x)=\frac{2c_{2,0}(c_{0,1}(x)-c_{0,1}(x_0))-c_{1,1}(c_{1,0}(x)-c_{1,0}(x_0))}{4c_{2,0}c_{0,2}-c_{1,1}^2}$ and
$c_3(x)=\frac{(c_{0,0}(x)-c_{0,0}(x_0))-c_{1,0}(x_0)c_1(x)-c_{0,1}(x_0)c_2(x)}{2c_{2,0}+2c_{0,2}+c_{1,1}}$. It is clear that $c_1(x)$, $c_2(x)$ and $c_3(x)$ vanish at $x=x_0$.
\end{proof}

\section{Conditions of local solvability for operators of bounded perturbation type in the Colombeau context}
\label{sec_local}
Purpose of this section is to provide sufficient conditions of local solvability for an operator $P(x,D)=\sum_{|\alpha|\le m}c_\alpha(x)D^\alpha$ with coefficients in $\G(\R^n)$ which is of bounded perturbation type in a neighborhood $\Om$ of $x_0$, that is 
\[
P(x,D)=P_0(D)+\sum_{j=1}^r c_j(x)P_j(D)
\]
on $\Om$ with,
\begin{itemize}
\item[(h1)] for all $j=1,...,n$, $c_j\in\G(\Om)$, $c_j(x_0)=0$, $P_0(D)$ and $P_j(D)$ operators with constant Colombeau coefficients
\item[(h2)]$\wt{P_0}$ invertible in some point of $\R^n$,
\item[(h3)]$P_0(D)$ stronger than any $P_j(D)$.
\end{itemize}
This requires some specific properties of the spaces $B_{p,k}$ which are proven in \cite[Chapters X, XIII]{Hoermander:V2} and collected in the sequel.
\subsection{The spaces $B_{p,k}$: properties and calculus}
Given $k\in\mathcal{K}$ we define for any $\nu>0$ the functions
\beq
\label{def_k_delta}
k_\nu(\xi)=\sup_\eta \esp^{-\nu|\eta|}k(\xi-\eta)
\eeq
and
\[
M_k(\xi)=\sup_\eta k(\xi+\eta)/k(\eta).
\]
One easily proves that $k_\nu$ and $M_k$ are both tempered weight functions. More precisely there exists a constant $C_\nu>0$ such that for all $\xi\in\R^n$, 
\[
1\le k_\nu(\xi)/k(\xi)\le C_\nu.
\]
If $k(\xi+\eta)\le(1+C|\xi|)^Nk(\eta)$ then 
\[
M_{k_\nu}(\xi)\le (1+C|\xi|)^N,
\]
for all $\nu>0$. In particular $M_{k_\nu}\to 1$ uniformly on compact subsets of $\R^n$ when $\nu\to 0$. The following theorem collects some important properties of the spaces $B_{p,k}$ which are proven in \cite[Chapters X, XIII]{Hoermander:V2}.

\begin{theorem}
\label{theo_Bpk}
\leavevmode
\begin{itemize}
\item[(i)] If $u_1\in B_{p,k_1}\cap\E'$ and $u_2\in B_{\infty,k_2}$ then $u_1\ast u_2\in B_{p,k_1k_2}$ and 
\[
\Vert u_1\ast u_2\Vert_{p,k_1k_2} \le \Vert u_1\Vert_{p,k_1}\, \Vert u_2\Vert_{\infty,k_2}.
\]
\item[(ii)] If $u\in B_{p,k}$ and $\phi\in\S(\R^n)$ then $\phi u\in B_{p,k}$ and
\[
\Vert \phi u\Vert_{p,k}\le \Vert\phi\Vert_{1,M_k}\, \Vert u\Vert_{p,k}.
\]
\item[(iii)] For every $\phi\in\S(\R^n)$ there exists $\nu_0>0$ such that
\[
\Vert \phi u\Vert_{p,k_{\nu}}\le 2\Vert \phi\Vert_{1,1}\Vert u\Vert_{p,k_{\nu}}
\]
for all $\nu\in(0,\nu_0)$.
\item[(iv)] If $\psi\in\Cinfc(\R^n)$, $x_0\in\R^n$ and $h$ is a $\Cinf$-function with $h(x_0)=0$ then for $\psi_{\delta,x_0}(x)=\psi((x-x_0)/\delta)$ one has 
\[
\Vert \psi_{\delta,x_0}h\Vert_{1,1}=O(\delta)
\]
as $\delta\to 0$.
\end{itemize}
\end{theorem}
Note that the number $\nu_0$ in $(iii)$ depends on $\phi$ and the weight function $k$ while the norm of the operator $u\to\phi u$ does not depend on $k$.

The next proposition concerns nets of distributions and nets of $B_{p,k}$-elements.
\begin{proposition}
\label{prop_basic}
\leavevmode
\begin{itemize}
\item[(i)] If $(g_\eps)_\eps\in\E'(\R^n)^{(0,1]}$ generates a basic functional $\LL(\G(\R^n),\wt{\C})$ then for all $p\in[1,+\infty]$ there exists $k\in\mathcal{K}$ such that $(g_\eps)_\eps\in\mM_{B_{p,k}(\R^n)}$; in particular if $(g_\eps)_\eps$ is the representative of a generalized function in $\Gc(\R^n)$ then $(g_\eps)_\eps\in\mM_{B_{p,k}(\R^n)}$ for all $k$.
\item[(ii)] If $(g_\eps)_\eps\in\E'(\R^n)^{(0,1]}$ with $\supp\, g_\eps\subseteq K\Subset\R^n$ for all $\eps$ and $(g_\eps)_\eps\in\mM_{B_{p,k}(\R^n)}$ then $(g_\eps)_\eps$ generates a basic functional in $\LL(\G(\R^n),\wt{\C})$.
\item[(iii)] If $(g_\eps)_\eps\in\mM_{B_{p,k}(\R^n)}$ and $\lara{\xi}^j k^{-1}(\xi)\in L^{q}(\R^n)$ with $1/p+1/q=1$ then $(g_\eps)_\eps\in\mM_{\mathcal{C}^j(\R^n)}$.
\item[(iv)] If $(S_\eps)_\eps$ and $(T_\eps)_\eps$ generate basic functionals in $\LL(\G(\R^n),\wt{\C})$ then $(S_\eps\ast T_\eps)_\eps$ defines a basic functional too.
\end{itemize}
\end{proposition}
\begin{proof}
$(i)$ The definition of a basic functional implies that there exists a moderate net $(c_\eps)_\eps$, a compact set $K\Subset\R^n$ and a number $m\in\N$ such that the estimate
\[
|\widehat{g_\eps}(\xi)|=|g_\eps(\esp^{-i\cdot\xi})|\le c_\eps\sup_{x\in K, |\alpha|\le m}|\partial^\alpha_x\esp^{-ix\xi}|\le c_\eps\lara{\xi}^m
\]
holds for all $\eps\in(0,1]$ and $\xi\in\R^n$. For $k(\xi):=\lara{\xi}^{\frac{-mp-n-1}{p}}$ we obtain
\[
\Vert g_\eps\Vert_{p,k}^p=\int_{\R^n}\lara{\xi}^{-mp-n-1}|\widehat{g_\eps}(\xi)|^p\, \dslash\xi \le {c'_\eps}\int_{\R^n}\lara{\xi}^{-mp-n-1}\lara{\xi}^{mp}\, d\xi.
\]
Thus $(g_\eps)_\eps\in\mM_{B_{p,k}(\R^n)}$. The second assertion of $(i)$ is clear since if $(g_\eps)_\eps$ is the representative of a generalized function in $\Gc(\R^n)$ then $(\widehat{g_\eps})_\eps$ is a moderate net of functions in $\S(\R^n)$ and clearly a moderate net of $B_{p,k}$-functions for all $k$.

$(ii)$ Let $f\in\Cinf(\R^n)$ and $\psi$ be a cut-off function identically $1$ in a neighborhood of $K$. We can write $g_\eps(f)$ as $\widehat{g_\eps}((\psi f)\check\,)=k\widehat{g_\eps}(k^{-1}(\psi f)\check\,)$. Hence,  
\[
|g_\eps(f)|\le \Vert k\widehat{g_\eps}\Vert_p\Vert k^{-1}(\psi f)\check\,\Vert_q,
\]
with $1/p+1/q=1$. By choosing $h$ large enough, using the bound from below $k(\xi)\ge k(0)(1+C|\xi|)^{-N}$ and the continuity of the inverse Fourier transform on $\S(\R^n)$ we are led to
\[
|g_\eps(f)|\le c\Vert k\widehat{g_\eps}\Vert_p\, \sup_{\xi\in\R^n, |\alpha|\le h}\lara{\xi}^h|\partial^\alpha((\psi f)\check\,)|
\]
and, for some $h'\in\N$ and $c'>0$, to
\[
|g_\eps(f)|\le c'\Vert k\widehat{g_\eps}\Vert_p\sup_{x\in \supp\,\psi, |\beta|\le h'}|\partial^\beta f(x)|.
\]
Since $(g_\eps)_\eps$ is $B_{p,k}$-moderate it follows that $(g_\eps)_\eps$ defines a basic functional in $\LL(\G(\R^n),\wt{\C})$.

$(iii)$ From the hypothesis $\lara{\xi}^j k^{-1}(\xi)\in L^{q}(\R^n)$ and $(g_\eps)_\eps\in\mM_{B_{p,k}(\R^n)}$ it follows that $(\xi^\alpha \widehat{g_\eps})_\eps\in\mM_{L^1(\R^n)}$ for all $\alpha$ with $|\alpha|\le j$. Therefore, $g_\eps(x)=\int_{\R^n}\esp^{ix\xi}\widehat{g_\eps}(\xi)\, \dslash\xi$ is a moderate net of $\mathcal{C}^j$-functions.

$(iv)$ Combining the property of basic functionals with the definition of convolution we get
\begin{multline*}
|S_\eps\ast T_\eps(f)|=|S_{\eps,x}T_{\eps,y}(f(x+y))|\le c_\eps\sup_{x\in K, |\alpha|\le m}|\partial^\alpha_x(T_{\eps,y}(f(x+y)))|=c_\eps\sup_{x\in K, |\alpha|\le m}|T_{\eps,y}(\partial^\alpha f(x+y))|\\
\le c_\eps\sup_{x\in K, |\alpha|\le m}c'_\eps\sup_{y\in K', |\beta|\le m'}|\partial^{\alpha+\beta}f(x+y)|\le c_\eps c'_\eps\sup_{z\in K+K', |\gamma|\le m+m'}|\partial^\gamma f(z)|,
\end{multline*}
valid for all $f\in\Cinf(\R^n)$ and for all $\eps\in(0,1]$ with $(c_\eps)_\eps$ and $(c'_\eps)_\eps$ moderate nets.
\end{proof}
In the course of the paper we will use the expression basic functional $T\in\LL(\G(\R^n),\wt{\C})$ of order $N$ for a functional $T$ defined by a net of distributions $(T_\eps)_\eps\in\E'(\R^n)^{(0,1]}$ such that 
\[
|T_\eps(f)|\le \lambda_\eps\sup_{x\in K\Subset\R^n, |\alpha|\le N}|\partial^\alpha f(x)|
\]
holds for all $f\in\Cinf(\R^n)$, for all $\eps\in(0,1]$ and for some moderate net $(\lambda_\eps)_\eps$. It follows from Proposition \ref{prop_basic}$(i)$ that $(T_\eps)_\eps$ is a moderate net in $B_{p,k}(\R^n)^{(0,1]}$ with $k(\xi)=\lara{\xi}^{\frac{-Np-n-1}{p}}$.

In our investigation of the local solvability of $P(x,D)$ we distinguish between 
\begin{enumerate}
\item[(1)] the coefficients $c_j$ are standard smooth functions,
\item[(2)] the coefficients $c_j$ are Colombeau generalized functions.
\end{enumerate}
In both these cases we will adapt the proof of Theorem 13.3.3. in \cite{Hoermander:V2} (or Theorem 7.3.1 in \cite{Hoermander:63}) to our generalized context. From the assumption of BP-type in $\Om$ it is not restrictive in the following statements to take $P(x,D)=P_0(D)+\sum_{j=1}^r c_j(x)P_j(D)$ with coefficients $c_j\in\G(\Om)$. 
\subsection{Local solvability: case $c_j\in\Cinf$.}
\begin{theorem}
\label{theo_locsolv_easy}
Let $\Om$ be a neighborhood of $x_0$ and let $P(x,D)=P_0(D)+\sum_{j=1}^r c_j(x)P_j(D)$ with $c_j\in\Cinf(\Om)$ for all $j$. If the hypotheses $(h1)$, $(h2)$, $(h3)$ are fulfilled and in addition  
\begin{itemize}
\item[(h4)] $\wt{P_{j,\eps}}(\xi)\le\lambda_{j,\eps}\wt{P_{0,\eps}}(\xi)$, where $\lambda_{j,\eps}=O(1)$
\end{itemize} 
holds for all $j$ and for a certain choice of representatives, then there exists a sufficiently small neighborhood $\Om_{\delta}:=\{x:\, |x-x_0|<\delta\}$ of $x_0$ such that
\begin{itemize}
\item[(i)] for all $F\in\LLb(\G(\R^n),\wt{\C})$ there exists $T\in\LLb(\G(\R^n),\wt{\C})$ solving $P(x,D)T=F$ on $\Om_\delta$,
\item[(ii)] for all $v\in\Gc(\R^n)$ there exists $u\in\Gc(\R^n)$ solving $P(x,D)u=v$ on $\Om_{\delta}$.
\end{itemize}
\end{theorem}
\begin{proof}
We organize the proof in few steps.

\bf{Step 1: the operator $P_0(D)$}\rm\\
Since $\wt{P_0}$ is invertible in some point of $\R^n$ from Theorem \ref{theo_fund_P} we know that there exists a fundamental solution $E_0\in\LLb(\Gc(\R^n),\wt{\C})$ having a representative $(E_{0,\eps})_\eps$ such that
\[
\biggl\Vert \frac{E_{0,\eps}}{\cosh(c|x|)}\biggr\Vert_{\infty,\wt{P_{0,\eps}}}\le C_0,
\]
for all $\eps\in(0,1]$. Note that the constant $C_0$ does not depend on $\eps$. It follows from Theorem \ref{theo_Bpk}$(ii)$, the inequality \eqref{est_Hoer} and the definition of $M_{\wt{P_{0,\eps}}}$ the estimate
\begin{multline*}
\Vert\varphi E_{0,\eps}\Vert_{\infty,\wt{P_{0,\eps}}}\le\Vert\varphi\cosh(c|x|)\Vert_{1,M_{\wt{P_{0,\eps}}}}\biggl\Vert \frac{E_{0,\eps}}{\cosh(c|x|)}\biggr\Vert_{\infty,\wt{P_{0,\eps}}}\\
\le \Vert (1+C|\xi|)^m\mathcal{F}(\varphi\cosh(c|\cdot|))(\xi)\Vert_1\, \cdot C_0\le C_1
\end{multline*}
valid for all $\varphi\in\Cinfc(\R^n)$ and $\eps\in(0,1]$ with $m$ order of the polynomial $P_0$.

\bf{Step 2: the equation $P_{0,\eps}(D)u=f$ when $f\in\E'(\R^n)$ with $\supp\, f\subseteq \Om_{\delta_0}$}\rm\\
Let $\delta_0>0$ and let $\chi$ be a function in $\Cinfc(\R^n)$ identically $1$ in a neighborhood of $\{x:\, |x|\le 2\delta_0\}$. From the previous considerations we have that $F_{0,\eps}:=\chi E_{0,\eps}\in B_{\infty,\wt{P_{0,\eps}}}$ with $\Vert F_{0,\eps}\Vert_{\infty,\wt{P_{0,\eps}}}\le C_1$ for all $\eps$. Moreover, for all $f\in\E'(\R^n)$ with $\supp\, f\subseteq \Om_{\delta_0}$ we have that
\[
E_{0,\eps}\ast f = F_{0,\eps}\ast f
\]
on $\Om_{\delta_0}$. Hence, by definition of a fundamental solution of $P_0(D)$, we can write on $\Om_{\delta_0}$,
\beq
\label{first_part}
P_{0,\eps}(D)(F_{0,\eps}\ast f)=f.
\eeq
\bf{Step 3: the operator $\sum_{j=1}^r c_j(x)P_j(D)$ on $\Om_{\delta}\subseteq\Om_{\delta_0}$}\rm\\
We now study the operator 
\[
\sum_{j=1}^r c_j(x)P_j(D).
\]
Let $\psi\in\Cinfc(\R^n)$ such that $\psi(x)=1$ when $|x|\le 1$ and $\psi(x)=0$ when $|x|\ge 2$. We set $\psi_{\delta,x_0}(x)=\psi((x-x_0)/\delta)$. We fixed the representatives $(\wt{P_{0,\eps}})_\eps$, $(\wt{P_{j,\eps}})_\eps$ and $(\lambda_{j,\eps})_\eps$ fulfilling $(h4)$ and we study the net of operators
\[
A_{\delta,\eps}(g)=\sum_{j=1}^r \psi_{\delta,x_0}c_j P_{j,\eps}(D)(F_{0,\eps}\ast g)
\]
defined for $g\in\D'(\R^n)$. More precisely for $k\in\mathcal{K}$, $1\le p\le\infty$ and $k_\nu$ as in \eqref{def_k_delta} we want to estimate $A_{\delta,\eps}$ on $B_{p,k_\nu}$. Since $\psi_{\delta,x_0}c_j$ belongs to $\S(\R^n)$ by Theorem \ref{theo_Bpk}$(iii)$ we find a sufficiently small $\nu_\delta$, depending on the coefficients $c_j$ and on $\psi_{\delta,x_0}$, such that 
\[
\Vert \psi_{\delta,x_0}c_j P_{j,\eps}(D)(F_{0,\eps}\ast g)\Vert_{p,k_\nu}\le 2\Vert\psi_{\delta,x_0}c_j\Vert_{1,1}\, \Vert P_{j,\eps}(D)(F_{0,\eps}\ast g)\Vert_{p,k_\nu} 
\]
holds for all $\nu<\nu_\delta$. Hence, from Theorem \ref{theo_Bpk}$(i)$, the properties of the net $(F_{0,\eps})_\eps$ and $(h4)$ we have
\begin{multline*}
\Vert A_{\delta,\eps}(g)\Vert_{p,k_\nu}\le \sum_{j=1}^r 2\Vert\psi_{\delta,x_0}c_j\Vert_{1,1}\, \Vert P_{j,\eps}(D)(F_{0,\eps}\ast g)\Vert_{p,k_\nu}\le \sum_{j=1}^r 2\Vert\psi_{\delta,x_0}c_j\Vert_{1,1}\Vert P_{j,\eps}(D)F_{0,\eps}\Vert_{\infty,1}\Vert g\Vert_{p,k_\nu}\\
=\sum_{j=1}^r 2\Vert\psi_{\delta,x_0}c_j\Vert_{1,1}\Vert P_{j,\eps}\widehat{F_{0,\eps}}\Vert_{\infty}\Vert g\Vert_{p,k_\nu}\le \sum_{j=1}^r 2\Vert\psi_{\delta,x_0}c_j\Vert_{1,1}\Vert \wt{P_{j,\eps}}\widehat{F_{0,\eps}}\Vert_{\infty}\Vert g\Vert_{p,k_\nu}\\
\le 2\sum_{j=1}^r\Vert\psi_{\delta,x_0}c_j\Vert_{1,1}\,\lambda_{j,\eps}\,\Vert F_{0,\eps}\Vert_{\infty,\wt{P_{0,\eps}}}\Vert g\Vert_{p,k_\nu}\le 2C_1\sum_{j=1}^r\Vert\psi_{\delta,x_0}c_j\Vert_{1,1}\,\lambda_{j,\eps}\, \Vert g\Vert_{p,k_\nu}
\end{multline*}
Since $c_j(x_0)=0$ the assumptions of Theorem \ref{theo_Bpk}$(iv)$ are satisfied. Hence, $\Vert\psi_{\delta,x_0}c_j\Vert_{1,1}=O(\delta)$. Combining this fact with $|\lambda_{j,\eps}|=O(1)$ we conclude that there exist $\delta_1$ and $\epsilon_1$ small enough such that 
\[
\Vert A_{\delta,\eps}(g)\Vert_{p,k_\nu}\le 2^{-1}\Vert g\Vert_{p,k_\nu}
\]
is valid for all $\delta<\delta_1$, for all $\nu<\nu_\delta$, for all $g\in B_{p,k_\nu}(\R^n)$ and for all $\eps\in(0,\epsilon_1)$. Since $B_{p,k}=B_{p,k_\nu}$ it follows that for all $f\in B_{p,k}(\R^n)$ there exists a unique solution $(g_\eps)_{\eps\in(0,\eps_1)}$ with $g_\eps\in B_{p,k}(\R^n)$ of the equation
\[
g+A_{\delta,\eps}(g)=\psi_{\delta,x_0}f
\]
for $\eps\in(0,\eps_1)$. Note that both $\delta_1$ and $\eps_1$ do not depend on the weight function $k$ and that this is possible thanks to the equivalent norm $\Vert\cdot\Vert_{p,k_\nu}$.  

\bf{Step 4: the equation $P(x,D)T=F$ on $\Om_\delta$ with $F\in\LLb(\G(\R^n),\wt{\C})$}\rm\\
Let $(F_\eps)_\eps$ be a net in $\E'(\R^n)^{(0,1]}$ which defines $F$. By Proposition \ref{prop_basic}$(i)$ we know that $(F_\eps)_\eps\in \mM_{B_{p,k}(\R^n)}$ for some $k$. Let us take $\Om_\delta$ with $\delta<\delta_1<\delta_0$ such that the previous arguments are valid for $B_{p,k_\nu}=B_{p,k}$ with $\nu<\nu_\delta$. We study the equation at the level of representatives on $\Om_\delta$ which can be written as
\[
P_{0,\eps}(D)(T_\eps)+\sum_{j=1}^r \psi_{\delta,x_0}c_j P_{j,\eps}(D)T_\eps = \psi_{\delta,x_0}F_\eps.
\]
Let $(g_\eps)_\eps$ be the unique solution of the equation
\[
g_\eps+A_{\delta,\eps}(g_\eps)=\psi_{\delta,x_0}F_\eps
\]
on the interval $(0,\eps_1)$. We have that $(g_\eps)_\eps$ is $B_{p,k}$-moderate (for simplicity we can set $g_\eps=0$ for $\eps\in[\eps_1,1]$). Indeed,
\[
\Vert g_\eps\Vert_{p,k_\nu}\le \Vert A_{\delta,\eps}(g_\eps)\Vert_{p,k_\nu}+\Vert\psi_{\delta,x_0}F_\eps\Vert_{p,k_\nu}\le 2^{-1}\Vert g_\eps\Vert_{p,k_\nu}+\Vert\psi_{\delta,x_0}F_\eps\Vert_{p,k_\nu}
\] 
and
\beq
\label{est_g}
\Vert g_\eps\Vert_{p,k}\le c\Vert\psi_{\delta,x_0}F_\eps\Vert_{p,k}\le c\Vert\psi_{\delta,x_0}\Vert_{1,M_{k}}\Vert F_\eps\Vert_{p,k}\le \lambda_\eps, 
\eeq
with $(\lambda_\eps)_\eps\in\EM$. Since  $\supp\,{A_{\delta,\eps}(g)}\subseteq\supp\psi_{\delta,x_0}\subseteq \Om_{2\delta}$ for all $g\in\D'(\R^n)$ we conclude that $\supp\, g_\eps$ is contained in a compact set uniformly with respect to $\eps$ (and therefore $\supp\, g_\eps\subseteq\Om_{\delta_0}$ for some $\delta_0$). From Proposition \ref{prop_basic}$(ii)$ it follows that $(g_\eps)_\eps$ generates a basic functional in $\LL(\G(\R^n),\wt{\C})$. Let now 
\[
T_\eps= F_{0,\eps}\ast g_\eps.
\]
The fourth assertion of Proposition \ref{prop_basic} yields that the net $(T_\eps)_\eps$ defines $T\in\LLb(\G(\R^n),\wt{\C})$. By construction (steps 2 and 3) and for $\eps$ small enough, we have
\begin{multline*}
P_{0,\eps}(D)(T_\eps)|_{\Om_\delta}+\sum_{j=1}^r (\psi_{\delta,x_0}c_j P_{j,\eps}(D)T_\eps)|_{\Om_\delta} = P_{0,\eps}(D)(F_{0,\eps}\ast g_\eps)|_{\Om_\delta}+A_{\delta,\eps}(g_\eps)|_{\Om_\delta}\\
=g_\eps|_{\Om_\delta}+A_{\delta,\eps}(g_\eps)|_{\Om_\delta}=\psi_{\delta,x_0}F_\eps|_{\Om_\delta}=F_\eps|_{\Om_\delta}.
\end{multline*}
Hence $P(x,D)T|_{\Om_\delta}=F|_{\Om_\delta}$.

\bf{Step 5: the equation $P(x,D)u=v$ on $\Om_\delta$ with $v\in\Gc(\R^n)$}\rm\\
Let $(v_\eps)_\eps$ be a representative of $v$. By Proposition \ref{prop_basic}$(i)$ we know that we can work in the space $B_{p,k}(\R^n)$ with $k$ arbitrary. Moreover, the interval $(0,\eps_1)$ and the neighborhood $\Om_\delta$ do not depend on $k$. This means that we can write
\[
g_\eps+A_{\delta,\eps}(g_\eps)=\psi_{\delta,x_0}v_\eps,
\]
where, combining the moderateness of $\Vert g_\eps\Vert_{p,k}$ in \eqref{est_g} for any $k$ with Proposition \ref{prop_basic}$(iii)$, we have that $(g_\eps)_\eps\in \mM_{\Cinf(\R^n)}=\EM(\R^n)$. The convolution between a basic functional in $\LLb(\G(\R^n),\wt{\C})$ and a Colombeau generalized function in $\Gc(\R^n)$ gives a generalized function in $\Gc(\R^n)$ (see Propositions 1.12, 1.14 and Remark 1.16 in \cite{Garetto:06a}). Hence, $u$ with representative
\[
u_\eps:=F_{0,\eps}\ast g_\eps
\]
belongs to $\Gc(\R^n)$ and $P_\eps(x,D)u_\eps=\psi_{\delta,x_0}v_\eps =v_\eps$ on $\Om_\delta$ by construction.
\end{proof}
\begin{example}
On $\R^2$ we define the operator
\[
\eps^a D^2_x-\eps^bD^2_y+c_1(x,y)\eps^a D_x+c_2(x,y)\eps^b D_y+c_3(x,y)\eps^c,
\]
where $c_1$, $c_2$ and $c_3$ are smooth functions and $a,b,c\in\R$ with $c\ge\min\{a,b\}$. We set $P_{0,\eps}(D)=\eps^a D^2_x-\eps^bD^2_y$, $P_{1,\eps}(D)=\eps^a D_x$, $P_{2,\eps}(D)=\eps^b D_y$ and $P_{3,\eps}(D)=\eps^c$. Assume that all the coefficients $c_j$ vanish in a point $(x_0,y_0)$. The hypotheses of Theorem \ref{theo_locsolv_easy} are fulfilled. Indeed, 
\[
\wt{P_{0,\eps}}^2(\xi_1,\xi_2)=(\eps^a\xi_1^2-\eps^b\xi_2^2)^2+(2\eps^a\xi_1)^2+(2\eps^b\xi_2)^2+4\eps^{2a}+4\eps^{2b} 
\]
is invertible in $(1,0)$ and concerning the functions $\wt{P_{1,\eps}}^2(\xi_1)=\eps^{2a}\xi_1^2+\eps^{2a}$, $\wt{P_{2,\eps}}^2(\xi_2)=\eps^{2b}\xi_2^2+\eps^{2b}$ and $\wt{P_{3,\eps}}^2=\eps^{2c}$ the following inequalities hold:
\begin{multline*}
\wt{P_{1,\eps}}^2(\xi_1)=\eps^{2a}\xi_1^2+\eps^{2a}\le 4\eps^{2a}\xi_1^2+4\eps^{2a}\le \wt{P_{0,\eps}}^2(\xi_1,\xi_2),\\
\wt{P_{2,\eps}}^2(\xi_2)=\eps^{2b}\xi_2^2+\eps^{2b}\le 4\eps^{2b}\xi_2^2+4\eps^{2b}\le \wt{P_{0,\eps}}^2(\xi_1,\xi_2),\\
\wt{P_{3,\eps}}^2=\eps^{2c}\le 4\eps^{2a}+4\eps^{2b}\le\wt{P_{0,\eps}}^2(\xi_1,\xi_2).
\end{multline*}
\end{example}

\subsection{Local solvability: case $c_j\in\G$.}

\begin{theorem}
\label{theo_locsolv}
Let $\Om$ be a neighborhood of $x_0$ and let $P(x,D)=P_0(D)+\sum_{j=1}^r c_j(x)P_j(D)$ with $c_j\in\G(\Om)$ for all $j$. 

If, for a certain choice of representatives, the hypotheses $(h1)$, $(h2)$, $(h3)$ are fulfilled and in addition,  
\begin{itemize}
\item[(h5)] $\wt{P_{j,\eps}}(\xi)\le\lambda_{j,\eps}\wt{P_{0,\eps}}(\xi)$ with 
\[
\sup_{|\alpha|\le n+1+\frac{n+1}{p}+N}\sup_{x\in\Om}|\partial^\alpha c_{j,\eps}(x)|\, \lambda_{j,\eps}=O(1)
\]
for some $p\in[1,\infty)$, for some $N\in\N$ and for all $j=1,...,r$,
\end{itemize} 
then there exists a sufficiently small neighborhood $\Om_{\delta}:=\{x:\, |x-x_0|<\delta\}$ of $x_0$ such that
\begin{itemize}
\item[(i)] for all $F\in\LLb(\G(\R^n),\wt{\C})$ of order $N$ there exists $T\in\LLb(\G(\R^n),\wt{\C})$ solving $P(x,D)T=F$ on $\Om_\delta$.
\end{itemize}

If, for a certain choice of representatives, the hypotheses $(h1)$, $(h2)$, $(h3)$ are fulfilled and in addition,  
\begin{itemize}
\item[(h6)] $\wt{P_{j,\eps}}(\xi)\le\lambda_{j,\eps}\wt{P_{0,\eps}}(\xi)$ with 
\[
\forall N\in\N\ \exists a>0\qquad\qquad\sup_{|\alpha|\le n+1+N}\sup_{x\in\Om}|\partial^\alpha c_{j,\eps}(x)|\, \lambda_{j,\eps}=O(\eps^a)\ \ \ \ \ \ \ \ \ \ \ \ \ \ \ \ \ \ \ \
\]
for all $j=1,...,r$, 
\end{itemize}
then there exists a sufficiently small neighborhood $\Om_{\delta}:=\{x:\, |x-x_0|<\delta\}$ of $x_0$ such that
\begin{itemize}
\item[(ii)] for all $F\in\LLb(\G(\R^n),\wt{\C})$ there exists $T\in\LLb(\G(\R^n),\wt{\C})$ solving $P(x,D)T=F$ on $\Om_\delta$.
\item[(iii)] for all $v\in\Gc(\R^n)$ there exists $u\in\Gc(\R^n)$ solving $P(x,D)u=v$ on $\Om_{\delta}$.
\end{itemize}
\end{theorem}
\begin{proof}
As for Theorem \ref{theo_locsolv_easy} we organize the proof in few steps. Clearly the steps $1$ and $2$ are as in the proof of Theorem \ref{theo_locsolv_easy} since $P_0(D)$ is an operator with constant Colombeau coefficients. New methods have to be applied to $\sum_{j=1}^r c_j(x)P_j(D)$ since the coefficients $c_j$ are Colombeau generalized functions. We work on a neighborhood $\Om=\Om_{2\delta_0}$ of $x_0$ and we use the notations introduced in proving Theorem \ref{theo_locsolv_easy}.

\bf{First set of hypotheses}\rm\\
We begin with the assumptions $(h1)$, $(h2)$, $(h3)$ and $(h5)$.
 
\bf{Step 3: the operator $\sum_{j=1}^r c_j(x)P_j(D)$ on $\Om_{\delta}\subseteq\Om_{\delta_0}$}\rm\\
We fix a choice of representatives $(c_{j,\eps})$ and the representatives $P_{j,\eps}$ and $P_{0,\eps}$ of the hypothesis. Let $k\in\mathcal{K}$ and $p\in[1,\infty)$ as in $(h5)$. We define the operator
\[
A_{\delta,\eps}(g)=\sum_{j=1}^r \psi_{\delta,x_0}\psi_{\delta_0,x_0}^2c_{j,\eps}P_{j,\eps}(D)(F_{0,\eps}\ast g).
\]
It maps $\D'(\R^n)$ into $\E'(\R^n)$. We want to estimate the $B_{p,k_\nu}$-norm of $A_{\delta,\eps}(g)$. We begin by observing that by Theorem \ref{theo_Bpk}$(iv)$ there exists $\delta_1<\delta_0$ such that 
\[
\Vert \psi_{\delta,x_0}\psi_{\delta_0,x_0}\Vert_{1,1}\le c\delta
\]
for all $\delta<\delta_1$. Choosing $\delta<\delta_1$ and $\phi=\psi_{\delta,x_0}\psi_{\delta_0,x_0}\in\S(\R^n)$ from Theorem \ref{theo_Bpk}$(iii)$ we obtain a certain $\nu_\delta$ such that for all $\nu\in(0,\nu_\delta)$ and for all $\eps\in(0,1]$ the following estimate holds:
\begin{multline*}
\Vert A_{\delta,\eps}(g)\Vert_{p,k_\nu}\le 2\sum_{j=1}^r\Vert \psi_{\delta,x_0}\psi_{\delta_0,x_0}\Vert_{1,1}\, \Vert\psi_{\delta_0,x_0}c_{j,\eps}P_{j,\eps}(D)(F_{0,\eps}\ast g)\Vert_{p,k_\nu}\\
\le 2c\delta\sum_{j=1}^r\Vert\psi_{\delta_0,x_0}c_{j,\eps}P_{j,\eps}(D)(F_{0,\eps}\ast g)\Vert_{p,k_\nu}.
\end{multline*}
We assume $k(\xi+\eta)\le (1+|\xi|)^{N+\frac{n+1}{p}}k(\eta)$ for all $\xi$. It follows that $M_{k_\nu}(\xi)\le (1+|\xi|)^{N+\frac{n+1}{p}}$ for all $\nu$. An application of the first two assertions of Theorem \ref{theo_Bpk} combined with $(h5)$ and the properties of $F_{0,\eps}$ entails, on a certain interval $(0,\eps_k)$ depending on $k$, the inequality
\begin{multline*}
\Vert A_{\delta,\eps}(g)\Vert_{p,k_\nu}\le 2c\delta\sum_{j=1}^r\Vert\psi_{\delta_0,x_0}c_{j,\eps}\Vert_{1,M_{k_\nu}}\Vert P_{j,\eps}(D)(F_{0,\eps}\ast g)\Vert_{p,k_\nu}\\
\le 2c\delta\sum_{j=1}^r\Vert(1+C|\xi|)^{N+\frac{n+1}{p}}\mathcal{F}(\psi_{\delta_0,x_0}c_{j,\eps})(\xi)\Vert_1\Vert P_{j,\eps}(D)(F_{0,\eps})\Vert_{\infty,1}\Vert g\Vert_{p,k_\nu}\\
\le 2c\delta\sum_{j=1}^r c(\psi_{\delta_0,x_0})\sup_{|\alpha|\le n+1+\frac{n+1}{p}+N}\sup_{x\in\Om}|\partial^\alpha c_{j,\eps}(x)|\lambda_{j,\eps}\Vert F_{0,\eps}\Vert_{\infty,\wt{P_{0,\eps}}}\Vert g\Vert_{p,k_\nu}\\
\le 2\, c_1(k,\delta_0,\psi)\delta \Vert g\Vert_{p,k_\nu}.
\end{multline*}
This means that for any $k$ there exist $\delta=\delta_k$ small enough and $\eps_k$ small enough such that
\[
\Vert A_{\delta,\eps}(g)\Vert_{p,k_\delta}\le 2^{-1}\Vert g\Vert_{p,k_\delta}
\]
holds for all $\eps\in(0,\eps_k)$ and for all $g\in B_{p,k_\delta}(\R^n)=B_{p,k}(\R^n)$. It follows that for all $f\in B_{p,k}(\R^n)$ there exists a unique solution $(g_\eps)_{\eps\in(0,\eps_k)}$ with $g_\eps\in B_{p,k}(\R^n)$ of the equation
\[
g+A_{\delta,\eps}(g)=\psi_{\delta,x_0} f
\]
for $\eps\in(0,\eps_k)$.

\bf{Step 4: the equation $P(x,D)T=F$ on $\Om_\delta$ with $F\in\LLb(\G(\R^n),\wt{\C})$}\rm\\
Let $(F_\eps)_\eps$ be a net in $\E'(\R^n)^{(0,1]}$ which defines $F$ of order $N$. By Proposition \ref{prop_basic}$(i)$ we know that $(F_\eps)_\eps\in \mM_{B_{p,k}(\R^n)}$ for $k(\xi)=\lara{\xi}^{-N-\frac{n+1}{p}}$. Hence, for all $\nu>0$ we have $M_{k_\nu}(\xi)\le (1+|\xi|)^{N+\frac{n+1}{p}}$ and we are under the hypothesis of the previous step. Let us take $\Om_\delta$ with $\delta=\delta_k$ such that the previous arguments are valid for $B_{p,k_\delta}=B_{p,k}$. At the level of representatives we can write  
\[
P_{0,\eps}(D)(T_\eps)+\sum_{j=1}^r \psi_{\delta,x_0}\psi_{\delta_0,x_0}^2c_{j,\eps} P_{j,\eps}(D)T_\eps = \psi_{\delta,x_0}F_\eps
\]
on $\Om_\delta$. Let $(g_\eps)_\eps$ be the unique solution of the equation
\[
g_\eps+A_{\delta,\eps}(g_\eps)=\psi_{\delta,x_0}F_\eps
\]
on the interval $(0,\eps_k)$. As in Theorem \ref{theo_locsolv_easy} one easily sees that $(g_\eps)_\eps$ is $B_{p,k}$-moderate. Since  $\supp\,{A_{\delta,\eps}(g)}\subseteq\supp\psi_{\delta,x_0}\subseteq \Om_{2\delta}$ for all $g\in\D'(\R^n)$ we conclude that $\supp\, g_\eps$ is contained in a compact set uniformly with respect to $\eps$ (and therefore $\supp\, g_\eps\subseteq\Om_{\delta_0}$ for some $\delta_0$). From Proposition \ref{prop_basic}$(ii)$ it follows that $(g_\eps)_\eps$ generates a basic functional in $\LL(\G(\R^n),\wt{\C})$. Let now 
\[
T_\eps= F_{0,\eps}\ast g_\eps.
\]
The fourth assertion of Proposition \ref{prop_basic} yields that the net $(T_\eps)_\eps$ defines $T\in\LLb(\G(\R^n),\wt{\C})$. By construction, for all $\eps\in(0,\eps_k)$ we have
\begin{multline*}
P_{0,\eps}(D)(T_\eps)|_{\Om_\delta}+\sum_{j=1}^r (\psi_{\delta,x_0}\psi_{\delta_0,x_0}^2c_{j,\eps} P_{j,\eps}(D)T_\eps)|_{\Om_\delta} = P_{0,\eps}(D)(F_{0,\eps}\ast g_\eps)|_{\Om_\delta}+A_{\delta,\eps}(g_\eps)|_{\Om_\delta}\\
=g_\eps|_{\Om_\delta}+A_{\delta,\eps}(g_\eps)|_{\Om_\delta}=\psi_{\delta,x_0}F_\eps|_{\Om_\delta}=F_\eps|_{\Om_\delta}.
\end{multline*}
We have solved the equation $P(x,D)T=F$ on a neighborhood of $x_0$ which depends on the weight function $k$ or in other words on the order of the functional $F$. 

\bf{Second set of hypotheses}\rm\\
We now assume that $(h1)$, $(h2)$, $(h3)$ and $(h6)$ hold and we prove that in this case one can find a neighborhood $\Om_\delta$ which does not depend on the weight function $k$. 

\bf{Step 3: the operator $\sum_{j=1}^r c_j(x)P_j(D)$ on $\Om_{\delta}\subseteq\Om_{\delta_0}$}\rm\\
Let $k$ be an arbitrary weight function. Choosing $\delta<\delta_1$ and any $\nu\in(0,\nu_\delta)$ our set of hypotheses combined with the properties of $F_{0,\eps}$ yields on an interval $(0,\eps_k)$ depending on $k$ the following estimates:
\begin{multline*}
\Vert A_{\delta,\eps}(g)\Vert_{p,k_\nu}\le 2c\delta\sum_{j=1}^r\Vert\psi_{\delta_0,x_0}c_{j,\eps}\Vert_{1,M_{k_\nu}}\Vert P_{j,\eps}(D)(F_{0,\eps}\ast g)\Vert_{p,k_\nu}\\
\le 2c\delta\sum_{j=1}^r\Vert(1+C|\xi|)^{N_k}\mathcal{F}(\psi_{\delta_0,x_0}c_{j,\eps})(\xi)\Vert_1\Vert P_{j,\eps}(D)(F_{0,\eps})\Vert_{\infty,1}\Vert g\Vert_{p,k_\nu}\\
\le 2c\delta\sum_{j=1}^r c(\psi_{\delta_0},x_0)\sup_{|\alpha|\le n+1+N_k}\sup_{x\in\Om}|\partial^\alpha c_{j,\eps}(x)|\lambda_{j,\eps}\Vert F_{0,\eps}\Vert_{\infty,\wt{P_{0,\eps}}}\Vert g\Vert_{p,k_\nu}\\
\le 2\delta\, c_1(k,\delta_0,\psi)\eps^{a_k} \Vert g\Vert_{p,k_\nu}.
\end{multline*}
At this point taking $\eps_k$ so small that $c_1(k,\delta_0,\psi)\eps^{a_k}_k<1$ and by requiring $\delta<1/4$ we have that for all $\delta<\min({\delta_1,1/4})$ there exist $\nu_\delta$ such that for all $\nu<\nu_\delta$ the inequality
\[
\Vert A_{\eps,\delta}(g)\Vert_{p,k_\nu}\le 2^{-1}\Vert g\Vert_{p,k_\nu}
\]
holds in a sufficiently small interval $\eps\in(0,\eps_k)$. Note that $\delta$ does not depend on $k$ while $\eps_k$ does. Clearly for all for all $f\in B_{p,k}(\R^n)$ there exists a unique solution $(g_\eps)_{\eps\in(0,\eps_k)}$ with $g_\eps\in B_{p,k}(\R^n)$ of the equation $g+A_{\delta,\eps}(g)=\psi_{\delta,x_0} f$ for $\eps\in(0,\eps_k)$.

\bf{Step 4: the equation $P(x,D)T=F$ on $\Om_\delta$ with $F\in\LLb(\G(\R^n),\wt{\C})$}\rm\\
Let $(F_\eps)_\eps$ be a net in $\E'(\R^n)^{(0,1]}$ which defines $F$. By Proposition \ref{prop_basic}$(i)$ we know that $(F_\eps)_\eps\in \mM_{B_{p,k}(\R^n)}$ for some weight function $k$. The previous arguments are valid in a neighborhood $\Om_\delta$ of $x_0$, with $\delta$ independent of $k$ and $(g_\eps)_\eps$ is the unique solution of the equation
\[
g_\eps+A_{\delta,\eps}(g_\eps)=\psi_{\delta,x_0}F_\eps
\]
on the interval $(0,\eps_k)$. The functional generated by
\[
T_\eps= F_{0,\eps}\ast g_\eps
\]
is the solution $T$ of the equation $P(x,D)T=F$ on $\Om_\delta$.
 
\bf{Step 5: the equation $\sum_{j=1}^r c_j(x)P_j(D)=v\in\Gc(\R^n)$ on $\Om_{\delta}$}\rm\\
Let $(v_\eps)_\eps\in\EM(\R^n)$ be a representative of $v\in\Gc(\R^n)$. For $\eps\in(0,\eps_k)$ the equation 
\beq
\label{eq_import}
A_{\eps,\delta}g+g=\psi_{\delta,x_0}v_\eps
\eeq
has a unique solution $g_\eps$ in $B_{p,k}$. Moreover, 
\beq
\label{est_import}
\Vert g_\eps\Vert_{p,k}\le C\Vert\psi_{\delta_0,x_0}v_\eps\Vert_{p,k}
\eeq
for all $\eps\in(0,\eps_k)$. Since $\psi_{\delta,x_0}v_\eps$ belongs to any $B_{p,k}$ space we can conclude that there exists a net of distributions $(g_\eps)_\eps$ which solve the equation \eqref{eq_import} in $\D'$ and such that for all $k\in\mathcal{K}$ the estimate \eqref{est_import} holds on a sufficiently small interval $(0,\eps_k)$. From the embedding $B_{p,k}\subset \mathcal{C}^j$ when $(1+|\xi|)^j/k(\xi)\in L^{q}$, $1/p+1/q=1$ the inequality \eqref{est_import} and the fact that $(\Vert\psi_{\delta,x_0}v_\eps\Vert_{p,k})_\eps$ is moderate we have the following:
\[
\forall j\in\N\, \exists \eps_j\in(0,1], \eps_j\searrow 0,\qquad (g_\eps)_{\eps\in(0,\eps_j)}\in\mM_{\mathcal{C}^j(\R^n)}. 
\]
In other words the net of distributions $(g_\eps)_\eps$ solves the equation \eqref{eq_import} in $\D'$ and has more and more moderate derivatives as $\eps$ goes to $0$. $(g_\eps)_\eps$ is the representative of a basic functional $g$ in $\LL(\G(\R^n),\wt{\C})$. We already know that $u=F_0\ast g\in\LLb(\G(\R^n),\wt{\C})$ solves the equation $P(x,D)u=v$ in $\LL(\Gc(\Om_\delta),\wt{\C})$. 

\bf{Step 6: $g$ belongs to $\Gc(\R^n)$ then $u\in\Gc(\R^n)$}\rm\\
We finally prove that $g$ is an element of $\Gc(\R^n)$. Since we already know that $g$ has compact support we just have to prove that $g$ belongs to $\G(\R^n)$. This will imply that $u\in\Gc(\R^n)$. We generate a representative of $g$ which belongs to $\EM(\R^n)$. Let $(n_\eps)_\eps\in\Neg$ with $n_\eps\neq 0$ for all $\eps$ and $\rho\in\Cinfc(\R^n)$ with $\int\rho=1$. For $\rho_{n_\eps}(x)=\rho(x/n_\eps)n_\eps^{-n}$, the net $g_\eps\ast \rho_{n_\eps}$ belongs to $\EM(\R^n)$. Indeed, $g_\eps\ast\rho_{n_\eps}\in\Cinf$ for each $\eps$ and taking $\eps$ small enough
\[
\sup_{x\in K}|\partial^\alpha (g_\eps\ast\rho_{n_\eps})(x)|= \sup_{x\in K}|\partial^\alpha (g_\eps)\ast\rho_{n_\eps}(x)|=\sup_{x\in K}\biggl|\int\partial^\alpha g_\eps(x-n_\eps z)\rho(z)\, dz\biggr|\le \eps^{-N}.
\]
Now, since
\[
g_\eps-g_\eps\ast\rho_{n_\eps}(x)=\int (g_\eps(x)-g_\eps(x-n_\eps z))\rho(z)\, dz
\]
and for all $q\in\N$,
\[
\sup_{x\in K}|g_\eps-g_\eps\ast\rho_{n_\eps}(x)|=O(\eps^q),
\]
we conclude that for all $u\in\Gc(\R^n)$, $g(u)=[(g_\eps\ast\rho_{n_\eps})_\eps](u)$ where  $[(g_\eps\ast\rho_{n_\eps})_\eps]\in\G(\R^n)$. This means that $g$ is an element of $\G(\R^n)$. 
\end{proof}
\begin{remark}
\label{rem_tech}
Note that in the proof we do not make use of the assertion $(iii)$ of Theorem \ref{theo_Bpk} with $\phi=\psi_{\delta,x_0}c_{j,\eps}$ and $u=P_{j,\eps}(D)(F_{0,\eps}\ast g)$ because this would generate some $\nu_0$ depending on $\phi$ and therefore on $\eps$. This would also lead to a neighborhood $\Om_\delta$ of $x_0$ whose radius depends on the parameter $\eps$. Finally, note that the condition $(h6)$ is an assumption of $\Ginf$-regularity for the coefficients $c_j$.
\end{remark}
\begin{example}
\label{example_2} 
On $\R^2$ we define the operator
\[
D^2_x-\eps^{-1}D^2_y+c_{1,\eps}(x,y)D_x+c_{2,\eps}(x,y)D_y+c_{3,\eps}(x,y),
\]
where $(c_{1,\eps})_\eps$, $(c_{2,\eps})_\eps$ and $(c_{3,\eps})_\eps$ are moderate nets of smooth functions. We set $P_{0,\eps}(D)=D^2_x-\eps^{-1}D^2_y$, $P_{1}(D)=D_x$, $P_{2}(D)=D_y$ and $P_{3}(D)=I$. Assume that the coefficients $c_{j,\eps}$ vanish in a point $(x_0,y_0)$ for all $\eps$. The weight function  
\[
\wt{P_{0,\eps}}^2(\xi_1,\xi_2)=(\xi_1^2-\eps^{-1}\xi_2^2)^2+(2\xi_1)^2+(2\eps^{-1}\xi_2)^2+4+4\eps^{-2} 
\]
is invertible in $(1,0)$ and concerning the functions $\wt{P_{1}}^2(\xi_1)=\xi_1^2+1$, $\wt{P_{2}}^2(\xi_2)=\xi_2^2+1$ and $\wt{P_{3}}^2=1$ the following inequalities hold:
\begin{multline*}
\wt{P_{1}}^2(\xi_1)=\xi_1^2+1\le 4\xi_1^2+4\le \wt{P_{0,\eps}}^2(\xi_1,\xi_2)\\
\wt{P_{2}}^2(\xi_2)=\xi_2^2+1=\frac{1}{4}\eps^2(4\eps^{-2}(\xi_2^2+1))\le \frac{1}{4}\eps^2\wt{P_{0,\eps}}^2(\xi_1,\xi_2)\\
\wt{P_{3}}^2=1\le 4\le \wt{P_{0,\eps}}^2(\xi_1,\xi_2).
\end{multline*}
One easily sees that for $\lambda_{1}=1$, $\lambda_{2,\eps}=\frac{1}{2}\eps$ and $\lambda_3=1$ the assumptions $(h1)$, $(h2)$, $(h3)$ and $(h6)$ of Theorem \ref{theo_locsolv} are fulfilled if, on a certain neighborhood $\Om$ of $x_0$ and for all $N\in\N$, the following holds: 
\[
\sup_{|\alpha|\le N}\sup_{x\in\Om}|\partial^\alpha c_{1,\eps}(x)|=O(\eps),\quad \sup_{|\alpha|\le N}\sup_{x\in\Om}|\partial^\alpha c_{2,\eps}(x)|=O(1),\quad \sup_{|\alpha|\le N}\sup_{x\in\Om}|\partial^\alpha c_{3,\eps}(x)|=O(\eps).
\]
Hence the operator
\[
D^2_x-[(\eps^{-1})_\eps]D^2_y+c_1(x,y)D_x+c_2(x,y)D_y+c_3(x,y),
\]
with $\Ginf$-coefficients $c_1=[(c_{1,\eps})_\eps]$, $c_2=[(c_{2,\eps})_\eps]$ and $c_3=[(c_{3,\eps}(x,y))_\eps]$ 
is locally solvable at $x_0$ in both the Colombeau algebra $\G(\R^n)$ and the dual $\LL(\Gc(\R^n),\wt{\C})$.
\end{example}

\section{A sufficient condition of local solvability for generalized pseudodifferential operators}
\label{sec_pseudo}
We conclude the paper by providing a sufficient condition of local solvability for generalized pseudodifferential operators. The generalization from differential to pseudodifferential operators obliges us to find suitable functional analytic methods able to generate a local solution which are not a simple convolution between the right-hand side term and a fundamental solution. In particular, we will make use of the Sobolev mapping properties of a generalized pseudodifferential operator $a(x,D)$. In the sequel we use the notation $H^0(\R^n)=L^2(\R^n)$ and $\Vert\cdot\Vert_s$ for the Sobolev $H^s$-norm.
\begin{proposition}
\label{prop_Sob}
Let $a\in\G_{S^m(\R^{2n})}$. There exist $l_0\in\N$ and a constant $C_0>0$ such that the inequality
\beq
\label{Sob_ineq}
\Vert a_\eps(x,D)u\Vert_{s-m}\le C_0\max_{|\alpha+\beta|\le l_0}|\lara{\xi}^{s-m}\sharp a_\eps\sharp\lara{\xi}^{-s}|^{(0)}_{\alpha,\beta}\,\Vert u\Vert_s
\eeq
holds for all $s\in\R$, for all $u\in H^s(\R^n)$, for all representatives $(a_\eps)_\eps$ of $a$ and for all $\eps\in(0,1]$.
\end{proposition}
It is clear from \eqref{Sob_ineq} that the moderateness properties of $(a_\eps)_\eps$ are the same of $(a_\eps(x,D)u)_\eps$ as a net in $H^{s-m}(\R^n)^{(0,1]}$.
\begin{remark}
\label{rem_Sob}
The Colombeau space $\G_{L^2(\R^n)}$ based on $L^2(\R^n)$ is not contained in the dual $\LL(\Gc(\R^n),\wt{\C})$. Indeed, for $f\in L^2(\R)\cap L^1(\R)$, $f\neq 0$, and $(n_\eps)_\eps\in\Neg$ we have that $f=[(n_\eps^{-1/2}f(\cdot/n_\eps))_\eps]$ is not $0$ in $\G_{L^2(\R)}$ but 
\[
\int_\R u(x)f(x)\, dx= 0
\]
for all $u\in\Gc(\R)$. Analogously, the embedding $H^{s_1}(\R^n)\subseteq H^{s_2}(\R^n)$ cannot be reproduced at the level of the corresponding Colombeau spaces for $s_1\ge s_2$. This is due to the fact that there exist nets in $\Neg_{H^{s_2}(\R^n)}\cap\mM_{H^{s_1}(\R^n)}$ which do not belong to $\Neg_{H^{s_1}(\R^n)}$. For example, for $f\in H^1(\R^n)$ with $\Vert f'\Vert_0\neq 0$ and $(n_\eps)_\eps\in\Neg$ we have that $(n_\eps^{1/2}f(\cdot/n_\eps))_\eps\in\mM_{H^1(\R)}\cap\Neg_{H^0(\R)}$ but $(n_\eps^{-1/2}f'(\cdot/n_\eps))_\eps\not\in\Neg_{H^0(\R)}$ and therefore $(n_\eps^{1/2}f(\cdot/n_\eps))_\eps\not\in\Neg_{H^1(\R)}$.
\end{remark}
The embedding issues of Remark \ref{rem_Sob} lead us to study the equation $a(x,D)T=F$ in the dual $\LL(\Gc(\R^n),\wt{\C}$ even when $F$ belongs to a Colombeau space based on a Sobolev space.
\begin{theorem}
\label{theo_loc_Sob}
Let $a\in\G_{S^m(\R^{2n})}$. Assume that there exist a representative $(a_\eps^\ast)_\eps$ of $a^\ast$, a strictly non-zero net $(\lambda_\eps)_\eps$, a positive number $\delta>0$ and $0\le s\le m$ such that
\beq
\label{inv_Sob}
\Vert\varphi\Vert_s\le \lambda_\eps\Vert a^\ast_\eps(x,D)|_{\Om_\delta}\varphi\Vert_0
\eeq
for all $\varphi\in\Cinfc(\Om_\delta)$, $\Om_\delta:=\{|x|<\delta\}$, and for all $\eps\in(0,1]$.\\
Then, for all $F\in\LLb(\Gc(\R^n),\wt{\C})$ generated by a net in $\mM_{H^{-s}(\R^n)}$ there exists $T\in\LLb(\Gc(\Om_\delta),\wt{\C})$ generated by a net in $\mM_{L^2(\Om_\delta)}$ such that
\[
a(x,D)|_{\Om_\delta}T=F|_{\Om_\delta}
\]
in $\LL(\Gc(\Om_\delta),\wt{\C})$.
\end{theorem}
The proof of Theorem \ref{theo_loc_Sob} makes use of the theory of Hilbert $\wt{\C}$-modules (see \cite{GarVer:08}) and in particular of the projection operators defined on the Hilbert $\wt{\C}$-module $\G_{L^2(\Om_\delta)}$. From Proposition 2.21 in \cite{GarVer:08} we have that if $E$ is a nonempty subset in $\G_{L^2(\Om_\delta)}$ generated by a net $(E_\eps)_\eps$ of nonempty convex subsets of $L^2(\Om_\delta)$ ($E=[(E_\eps)_\eps]:=\{u\in\G_{L^2(\Om_\delta)}:\, \exists\, {\rm{repr.}} (u_\eps)_\eps\, \forall\eps\in(0,1]\  u_\eps\in E_\eps\}$) then there exists a map $P_E:\G_{L^2(\Om_\delta)}\to E$ called projection on $E$ such that $\Vert u-P_E(u)\Vert=\inf_{w\in E}\Vert u-w\Vert$ for all $u\in\G_{L^2(\Om_\delta)}$. In this particular case we have in addition that the net $(P_{\overline{E}_\eps}\,(u_\eps))_\eps$ is moderate in $L^2(\Om_\delta)$ and the property $P_E(u)=[(P_{\overline{E}_\eps}\,(u_\eps))_\eps]$. If $E_\eps$ is a subspace of $L^2(\Om_\delta)$ we easily have for all $v\in L^2(\Om_\delta)$ the following inequality in the norm of $L^2(\Om_\delta)$: $\Vert P_{\overline{E}_\eps}v\Vert\le \Vert v- P_{\overline{E}_\eps}v\Vert+\Vert v\Vert\le \inf_{w\in\overline{E}_\eps}\Vert v-w\Vert+\Vert v\Vert\le 2\Vert v\Vert$.
\begin{proof}[Proof of Theorem \ref{theo_loc_Sob}]
Let $a_\eps^\ast$ be the representative of $a^\ast$ fulfilling the hypotheses of Theorem \ref{theo_loc_Sob}. Let $E_\eps:=\{\psi\in L^2(\Om_\delta):\, \exists\varphi\in\Cinfc(\Om_\delta)\quad \psi=a_\eps^\ast(x,D)|_{\Om_\delta}(\varphi)\}$. $E_\eps$ is a nonempty subspace of $L^2(\Om_\delta)$ and therefore we can define the projection operator $P_E:\G_{L^2(\Om_\delta)}\to E$ on $E:=[(E_\eps)_\eps]$ $\wt{\C}$-submodule of $\G_{L^2(\Om_\delta)}$. We use for the operator $\Cinfc(\Om_\delta)\to L^2(\Om_\delta):\varphi\to a_\eps^\ast(x,D)|_{\Om_\delta}(\varphi)$ the notation $A^\ast_\eps$. The condition \eqref{inv_Sob} means that $A^\ast_\eps:\Cinfc(\Om_\delta)\to E_\eps$ is invertible. Combining \eqref{inv_Sob} with the Sobolev embedding properties we have that 
\beq
\label{ext_A_eps}
\Vert (A^\ast_\eps)^{-1}v\Vert_{L^2(\Om_\delta)}\le\Vert (A^\ast_\eps)^{-1}v\Vert_s\le \lambda_\eps\Vert v\Vert_{L^2(\Om_\delta)}
\eeq
holds for all $v\in E_\eps$. Taking the closure $\overline{E}_\eps$ of $E_\eps$ in $L^2(\Om_\delta)$ the inequality \eqref{ext_A_eps} allows us to extend $(A^\ast_\eps)^{-1}$ to a continuous operator from $\overline{E}_\eps$ to $H^s(\R^n)\subseteq L^2(\Om_\delta)$, with \eqref{ext_A_eps} valid for all $v\in\overline{E}_\eps$. 

Let $u$ be an element of $E$ defined by the net $(u_\eps)_\eps$, $u_\eps\in E_\eps$. Clearly $(A^\ast_\eps)^{-1}u_\eps\in\Cinfc(\Om_\delta)$  with $\Vert(A^\ast_\eps)^{-1}u_\eps\Vert_s\le \lambda_\eps\Vert u_\eps\Vert_{L^2(\Om_\delta)}$ and if $u'_\eps\in E_\eps$ is another net generating $u$ we obtain
\[
\Vert(A^\ast_\eps)^{-1}(u_\eps-u'_\eps)\Vert_s\le \lambda_\eps\Vert u_\eps-u'_\eps\Vert_{L^2(\Om_\delta)}.
\]
This means that we can define the $\wt{\C}$-linear functional
\[
S:E\to\wt{\C}:u=[(u_\eps)_\eps]\to[(((A^\ast_\eps)^{-1}u_\eps|F_\eps)_{L^2(\R^n)})_\eps]
\]
where
\begin{multline*}
|((A^\ast_\eps)^{-1}u_\eps|F_\eps)_{L^2(\R^n)}|=(2\pi)^{-n}|(\lara{\xi}^{s}\widehat{(A^\ast_\eps)^{-1}u_\eps}|\lara{\xi}^{-s}\widehat{F_\eps})_{L^2(\R^n)}|\le \Vert (A^\ast_\eps)^{-1}u_\eps\Vert_{s}\Vert F_\eps\Vert_{-s}\\
\le \lambda_\eps\Vert u_\eps\Vert_{L^2(\Om_\delta)}\Vert F_\eps\Vert_{-s}.
\end{multline*}
From the previous inequality we also have that the functional $S$ is continuous. Since $E$ is a $\wt{\C}$-submodule the projection $P_E$ is continuous and $\wt{\C}$-linear. It follows that
\[
S\circ P_E:\G_{L^2(\Om_\delta)}\to \wt{\C}:u\to S(P_E(u))
\]
is a continuous $\wt{\C}$-linear functional on $\G_{L^2(\Om_\delta)}$ with basic structure. Indeed, it is defined by the net $L^2(\Om_\delta)\to \C: v\to ((A^\ast_\eps)^{-1}P_{\overline{E}_\eps}\ v|F_\eps)_{L^2(\R^n)}$ such that
\[
|((A^\ast_\eps)^{-1}P_{\overline{E}_\eps}\ v|F_\eps)_{L^2(\R^n)}|\le \lambda_\eps\Vert P_{\overline{E}_\eps}\ v\Vert_{L^2(\Om_\delta)}\Vert F_\eps\Vert_{-s}\le 2\lambda_\eps\Vert F_\eps\Vert_{-s}\Vert v\Vert_{L^2(\Om_\delta)},
\]
where the nets $(\lambda_\eps)_\eps$ and $(\Vert F_\eps\Vert_{-s})_\eps$ are moderate. By the Riesz representation theorem for Hilbert $\wt{\C}$-modules and $\wt{\C}$-linear functionals (Theorem 4.1 and Proposition 4.4 in \cite{GarVer:08}) we have that there exists a unique $t\in \G_{L^2(\Om_\delta)}$ such that
\[
(S\circ P_E)(u)=(u|t)_{L^2(\Om_\delta)}
\]
for all $u\in\G_{L^2(\Om_\delta)}$. More precisely there exists a representative $(t_\eps)_\eps$ of $t$ such that the equality
\[
((A^\ast_\eps)^{-1}P_{\overline{E}_\eps}\ v|F_\eps)_{L^2(\R^n)}=(v|t_\eps)_{L^2(\Om_\delta)}
\]
holds for all $v\in L^2(\Om_\delta)$. Let $T$ be the basic functional in $\LL(\Gc(\Om_\delta),\wt{\C})$ generated by the net $(t_\eps)_\eps$. $T$ solves the equation $a(x,D)|_{\Om_\delta}T=F|_{\Om_\delta}$. Indeed, since $A^\ast_\eps(x,D)\varphi\in E_\eps\subseteq L^2(\Om_\delta)$ for all $\varphi\in\Cinfc(\Om_\delta)$ we can write
\[
((A^\ast_\eps)^{-1}P_{\overline{E}_\eps}A^\ast_\eps(x,D)\varphi|F_\eps)_{L^2(\R^n)}=((A^\ast_\eps)^{-1}A^\ast_\eps(x,D)\varphi|F_\eps)_{L^2(\R^n)}=(\varphi|F_\eps)_{L^2(\R^n)}=(A^\ast_\eps(x,D)\varphi|t_\eps)_{L^2(\Om_\delta)}.
\]
Thus,
\[
(\varphi|F_\eps)_{L^2(\R^n)}=(\varphi|a_\eps(x,D)|_{\Om_\delta}t_\eps)_{L^2(\Om_\delta)}
\]
for all $\varphi\in\Cinfc(\Om_\delta)$ or in other words, $a_\eps(x,D)|_{\Om_\delta}t_\eps = F_\eps|_{\Om_\delta}$ in $\D'(\Om_\delta)$.
\end{proof}
Theorem \ref{theo_loc_Sob} tells us that classical operators which satisfy the condition \eqref{inv_Sob} are locally solvable in the Colombeau context, in the sense that under suitable moderateness conditions on the right hand side we will find a local generalized solution. An example is given by differential operators which are at the same time principally normal and of principal type at $0$. More precisely, Proposition 4.3 in \cite{SaintRaymond:91} proves that a differential operator $a(x,D)$ which is both principally normal and of principal type at $0$ fulfills the inequality $\Vert \varphi\Vert_{m-1}\le \Vert a^\ast(x,D)\varphi\Vert_0$ for all $\varphi\in\Cinfc(\Om_\delta)$ with $\delta$ small enough.

We now go back to the case of differential operators with generalized Colombeau coefficients. In other words we assume that the symbol $a$ is of the type $a=[(a_\eps)_\eps]=\sum_{|\alpha|\le m}c_\alpha(x)\xi^\alpha$ with $c_\alpha\in\G_\infty(\R^n)$. The next proposition shows that the local solvability condition \eqref{inv_Sob} holds under an ellipticity assumption on the real part of the symbol $a_\eps$. We recall that for all $m\in\N$ and for all $\delta>0$ the inequality
\beq
\label{ineq_SR}
\Vert\varphi\Vert_m\le 2\delta\Vert\varphi\Vert_{m+1}
\eeq
is true for all $\varphi\in\Cinfc(\Om_\delta)$ (see \cite[Lemma 4.2]{SaintRaymond:91}).
\begin{proposition}
\label{prop_SR}
Let $a(x,D)$ be a generalized differential operator with symbol $a\in\G_{S^{m}(\R^{2n})}$. If there exists $b\in\R$, a representative $(a_\eps)_\eps\in \mM_{S^{m}(\R^{2n}),b}$, a constant $c_0>0$ and a net $(c_\eps)_\eps\in \mM_{S^{m-1}(\R^{2n}),b}$ such that
\[
\Re\, a_\eps(x,\xi)= c_0\eps^b\lara{\xi}^{m}+c_\eps(x,\xi)
\]
for all $(x,\xi)$ and all $\eps\in(0,1]$, then there exist a sufficiently small $\delta>0$ and a constant $C>0$ such that
\[
\Vert \varphi\Vert_{\frac{m}{2}}\le C\eps^{-b}\Vert a_\eps^\ast(x,D)\varphi\Vert_0
\] 
for all $\varphi\in\Cinfc(\Om_\delta)$ and for all $\eps\in(0,1]$.
\end{proposition}
\begin{proof}
We begin by writing $2\Re\, (\varphi|a_\eps^\ast(x,D)\varphi)$ as $(\varphi|(a_\eps+a_\eps^\ast)(x,D)\varphi)$. Recalling that $a_\eps^\ast=\overline{a_\eps}$ modulo $\mM_{S^{m-1}(\R^n),b}$ we have
\[
2\Re\, (\varphi|a_\eps^\ast(x,D)\varphi)=(\varphi|2c_0\eps^b\lambda^m(D)\varphi)+(\varphi|c_{1,\eps}(x,D)\varphi),
\]
where $\lambda^m(D)$ has symbol $\lara{\xi}^m$ and $(c_{1,\eps})_\eps\in\mM_{S^{m-1}(\R^{2n}),b}$. Hence,
\[
2\Re\, (\varphi|a_\eps^\ast(x,D)\varphi)\ge c_1\eps^b\Vert\varphi\Vert^2_{\frac{m}{2}}-(\varphi|c_{1,\eps}(x,D)\varphi)\ge c_1\eps^b\Vert\varphi\Vert^2_{\frac{m}{2}}-\Vert\varphi\Vert_{\frac{m}{2}-1}\Vert c_{1,\eps}(x,D)\varphi\Vert_{-\frac{m}{2}}
\]
Combining Proposition \ref{prop_Sob} with \eqref{ineq_SR} we obtain
\[
2\Re\, (\varphi|a_\eps^\ast(x,D)\varphi)\ge c_1\eps^b\Vert\varphi\Vert^2_{\frac{m}{2}}-\Vert\varphi\Vert_{\frac{m}{2}-1}c_2\eps^b\Vert\varphi\Vert^2_{-\frac{m}{2}}\ge c_1\eps^b\Vert\varphi\Vert^2_{\frac{m}{2}}-2\delta c_2\eps^b\Vert\varphi\Vert^2_{\frac{m}{2}}.
\]
Concluding for $\delta$ small enough there exists a constant $C>0$ such that
\[
\Vert\varphi\Vert^2_{\frac{m}{2}}\le C\eps^{-b}\Vert\varphi\Vert_{\frac{m}{2}}\Vert a_\eps^\ast(x,D)\varphi\Vert_{-\frac{m}{2}}\le C\eps^{-b}\Vert\varphi\Vert_{\frac{m}{2}}\Vert a_\eps^\ast(x,D)\varphi\Vert_{0}
\]
holds for all $\varphi\in\Cinfc(\Om_\delta)$.
\end{proof}
\begin{example}
Note that the condition \eqref{inv_Sob} can be fulfilled by differential operators which are not a bounded perturbation of a differential operator with constant Colombeau coefficients. This means that the results of this section enlarge the family of generalized differential operators whose local solvability we are able to investigate in the Colombeau context. As an explanatory example in $\R^2$ consider 
\[
a_\eps(x,D)=D_1+b_\eps(x)D_2,
\]
where $(b_\eps)_\eps$ is the representative of a generalized function. The generalized differential operator $a(x,D)$ generated by $(a_\eps)_\eps$ is not a bounded perturbation of the operator at $0$ if we take $b_\eps(0)=0$. However if $(b_\eps)_\eps$ is real valued and suitable moderateness conditions are satisfied (for instance $(b_\eps)_\eps$ bounded in $\eps$ together with all its derivatives), the arguments of Proposition 4.3 in \cite{SaintRaymond:91} lead us to an estimate from below of the type considered by Theorem \ref{theo_loc_Sob}. 
\end{example}

\newcommand{\SortNoop}[1]{}


\begin{thebibliography}{10}

\bibitem{BO:92}
H.~Biagioni and M.~Oberguggenberger.
\newblock Generalized solutions to the {K}orteweg-{D}e {V}ries and the
  regularized long-wave equations.
\newblock {\em SIAM J. Math Anal.}, 23(4):923--940, 1992.

\bibitem{Colombeau:85}
J.~F. Colombeau.
\newblock {\em Elementary Introduction to New Generalized Functions}.
\newblock North-Holland Mathematics Studies 113. Elsevier Science Publishers,
  1985.

\bibitem{Garetto:04}
C.~Garetto.
\newblock Pseudo-differential operators in algebras of generalized functions
  and global hypoellipticity.
\newblock {\em Acta Appl. Math.}, 80(2):123--174, 2004.

\bibitem{Garetto:05b}
C.~Garetto.
\newblock Topological structures in {C}olombeau algebras: investigation of the
  duals of ${\Gc(\Om)}$, ${\G(\Om)}$ and ${\GS(\R^n)}$.
\newblock {\em Monatsh. Math.}, 146(3):203--226, 2005.

\bibitem{Garetto:05a}
C.~Garetto.
\newblock Topological structures in {C}olombeau algebras: topological
  $\wt{\C}$-modules and duality theory.
\newblock {\em Acta. Appl. Math.}, 88(1):81--123, 2005.

\bibitem{Garetto:06a}
C.~Garetto.
\newblock Microlocal analysis in the dual of a {C}olombeau algebra: generalized
  wave front sets and noncharacteristic regularity.
\newblock {\em New York J. Math.}, (12):275--318, 2006.

\bibitem{Garetto:08b}
C.~Garetto.
\newblock Fundamental solutions in the {C}olombeau framework: applications to
  solvability and regularity theory.
\newblock {\em Acta. Appl. Math.}, (102):281--318, 2008.

\bibitem{Garetto:ISAAC07}
C.~Garetto.
\newblock Generalized {F}ourier integral operators on spaces of {C}olombeau type.
\newblock {\em Operator Theory, Advances and Applications}, (189):137--184,
  2008.

\bibitem{Garetto:08c}
C.~Garetto.
\newblock $\mathcal{G}$ and $\mathcal{G}^\infty$-hypoellipticity of partial
  differential operators with constant {C}olombeau coefficients.
\newblock {\em To appear in Banach Center Publ.}, 2008.

\bibitem{GGO:03}
C.~Garetto, T.~Gramchev, and M.~Oberguggenberger.
\newblock Pseudodifferential operators with generalized symbols and regularity
  theory.
\newblock {\em Electron. J. Diff. Eqns.}, 2005(2005)(116):1--43, 2005.

\bibitem{GH:05}
C.~Garetto and G.~H\"{o}rmann.
\newblock Microlocal analysis of generalized functions: pseudodifferential
  techniques and propagation of singularities.
\newblock {\em Proc. Edinburgh. Math. Soc.}, 48(3):603--629, 2005.

\bibitem{GarVer:08}
C.~Garetto and H.~Vernaeve.
\newblock Hilbert $\wt{\C}$-modules: structural properties and applications to
  variational problems.
\newblock {\em arXiv:math. FA/07071104, to appear in Trans. Amer. Math. Soc.},
  2008.

\bibitem{GKOS:01}
M.~Grosser, M.~Kunzinger, M.~Oberguggenberger, and R.~Steinbauer.
\newblock {\em Geometric theory of generalized functions}, volume 537 of {\em
  Mathematics and its Applications}.
\newblock Kluwer, Dordrecht, 2001.

\bibitem{Hoermander:63}
L.~H{\"o}rmander.
\newblock {\em Linear Partial Differential Operators}.
\newblock Springer-Verlag, Berlin, 1963.

\bibitem{Hoermander:V2}
L.~H{\"o}rmander.
\newblock {\em The Analysis of Linear Partial Differential Operators},
  volume~II.
\newblock Springer-Verlag, 1983.

\bibitem{Hoermander:V1-4}
L.~H{\"o}rmander.
\newblock {\em The Analysis of Linear Partial Differential Operators}, volume
  I-IV.
\newblock Springer-Verlag, {\SortNoop{83}}1983-85, 2nd Ed. Vol. I 1990.

\bibitem{GH:03}
G.~H\"{o}rmann.
\newblock First-order hyperbolic pseudodifferential equations with generalized
  symbols.
\newblock {\em J. Math. Anal. Appl.}, 293(1):40--56, 2004.

\bibitem{HdH:01}
G.~H{\"o}rmann and M.~V. de~Hoop.
\newblock Microlocal analysis and global solutions of some hyperbolic equations
  with discontinuous coefficients.
\newblock {\em Acta Appl. Math.}, 67:173--224, 2001.

\bibitem{HdH:01c}
G.~H{\"o}rmann and M.~V. de~Hoop.
\newblock Geophysical modeling and regularity properties of {C}olombeau
  functions.
\newblock In A.~Delcroix, M.~Hasler, J.-A. Marti, and V.~Valmorin, editors,
  {\em Nonlinear Algebraic Analysis and Applications}, pages 99--110,
  Cambridge, 2004. Cambridge Scientific Publishers.

\bibitem{HO:03}
G.~H{\"o}rmann and M.~Oberguggenberger.
\newblock Elliptic regularity and solvability for partial differential
  equations with {C}olombeau coefficients.
\newblock {\em Electron. J. Diff. Eqns.}, 2004(14):1--30, 2004.

\bibitem{HOP:05}
G.~H\"{o}rmann, M.~Oberguggenberger, and S.~Pilipovic.
\newblock Microlocal hypoellipticity of linear partial differential operators
  with generalized functions as coefficients.
\newblock {\em Trans. Amer. Math. Soc.}, 358:3363--3383, 2006.

\bibitem{Kumano-go:81}
H.~Kumano-go.
\newblock {\em Pseudodifferential Operators}.
\newblock MIT Press, Cambridge, Mass., 1981.

\bibitem{LO:91}
F.~Lafon and M.~Oberguggenberger.
\newblock Generalized solutions to symmetric hyperbolic systems with
  discontinuous coefficients: the multidimensional case.
\newblock {\em J. Math. Anal. Appl.}, 160:93--106, 1991.

\bibitem{O:89}
M.~Oberguggenberger.
\newblock Hyperbolic systems with discontinuous coefficients: generalized
  solutions and a transmission problem in acoustics.
\newblock {\em J. Math. Anal. Appl.}, 142:452--467, 1989.

\bibitem{O:92}
M.~Oberguggenberger.
\newblock {\em Multiplication of Distributions and Applications to Partial
  Differential Equations}.
\newblock Pitman Research Notes in Mathematics 259. Longman Scientific {\&}
  Technical, 1992.

\bibitem{O:07}
M.~Oberguggenberger.
\newblock Hyperbolic systems with discontinuous coefficients: generalized
  wavefront sets.
\newblock {\em Operator Theory, Advances and Applications}, (189):117--136,
  2008.

\bibitem{PilSca:06}
S.~Pilipovi\'{c} and D.~Scarpalezos.
\newblock Divergent type quasilinear {D}irichlet problem with singularities.
\newblock {\em Acta Appl. Math.}, 94(1):67--82, 2006.

\bibitem{SaintRaymond:91}
X.~S. Raymond.
\newblock {\em Elementary Introduction to the Theory of Pseudo-differential
  Operators}.
\newblock Studies in Advanced Mathematics. CRC Press, Boca Raton, FL, 1991.

\bibitem{Yosida:80}
K.~Yosida.
\newblock {\em Functional Analysis}.
\newblock Springer Verlag, Berlin Heidelberg New York, 1980.

\end{thebibliography}
\end{document}